\documentclass{svjour3_opte}                     

\smartqed  
\newcommand{\dom}{\text{dom}}
\newcommand{\epi}{\text{epi}}

\newcommand{\N}{\mathbb{N}}
\newcommand{\R}{\mathbb{R}}

\newcommand{\InDCAe}{InDCA$_\text{e}$}
\newcommand{\InDCAine}{InDCA$_\text{n}$}
\newcommand{\RInDCAe}{RInDCA$_\text{e}$}
\newcommand{\RInDCAine}{RInDCA$_\text{n}$}
\newcommand{\vx}{\textbf{x}}
\newcommand{\vy}{\textbf{y}}
\newcommand{\vz}{\textbf{z}}
\newcommand{\vw}{\textbf{w}}
\newcommand{\zero}{\textbf{0}}
\newcommand{\bX}{\textbf{X}}
\newcommand{\orcid}[1]{\href{https://orcid.org/#1}{\includegraphics[scale=1]{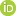}}}
\newcommand{\redtext}[1]{#1}

\usepackage{multirow}
\usepackage{subfigure}
\usepackage{graphicx}
\usepackage{amsmath,amsfonts}
\usepackage{amssymb}
\usepackage{cases}
\usepackage{multirow}
\usepackage{adjustbox}
\usepackage[graphicx]{realboxes}
\usepackage[figuresright]{rotating}
\usepackage{lscape}
\usepackage[colorlinks,linkcolor=blue,anchorcolor=blue,citecolor=blue,urlcolor=blue]{hyperref}
\usepackage[ruled,linesnumbered]{algorithm2e}
\usepackage{listings}
\graphicspath{{figs/}}

\DeclareMathOperator*{\argmin}{argmin}
\DeclareMathOperator*{\argmax}{argmax}

\newtheorem{assumption}{Assumption}

\begin{document}
\title{A Refined Inertial DC Algorithm for DC Programming \thanks{This work is supported by the National Natural Science Foundation of China (Grant 11601327).}
}

\author{Yu You \and Yi-Shuai Niu \orcid{0000-0002-9993-3681}}
\authorrunning{Yu You \and Yi-Shuai Niu} 

\institute{
           Yu You \at
           School of Mathematical Sciences, Shanghai Jiao Tong University, China \\
           \email{youyu0828@sjtu.edu.cn}
            \and
           Yi-Shuai Niu \at
            Department of Applied Mathematics, The Hong Kong Polytechnic University, Hong Kong \\
            School of Mathematical Sciences \& SJTU-Paristech, Shanghai Jiao Tong University, China \\
            \email{yi-shuai.niu@polyu.edu.hk, niuyishuai@sjtu.edu.cn}
}

\date{}

\maketitle

\begin{abstract}
In this paper we consider the  difference-of-convex (DC) programming problems, whose objective function is the difference of two convex functions. The classical DC Algorithm (DCA) is well-known for solving this kind of problems, which generally returns a critical point. Recently, an inertial DC algorithm (InDCA) equipped with heavy-ball inertial-force procedure was proposed in de Oliveira et al. (Set-Valued and Variational Analysis
27(4):895--919, 2019), which potentially helps to improve both the convergence speed and the solution quality. Based on InDCA, we propose a refined inertial DC algorithm (RInDCA) equipped with enlarged inertial step-size compared with InDCA. Empirically, larger step-size accelerates the convergence. We demonstrate the subsequential convergence of our refined version to a critical point. In addition, by assuming the Kurdyka-Łojasiewicz (KL) property of the objective function, we establish the sequential convergence of RInDCA. Numerical simulations on checking copositivity of matrices and image denoising problem show the benefit of larger step-size. 
\end{abstract}

\section{Introduction}
\label{intro}
Difference-of-convex (DC) programming, referring to the problems of minimizing a function which is the difference of two convex functions, forms an important class of nonconvex programming and has been studied extensively for decades, e.g., \redtext{ \cite{hiriart1985generalized,Lethi_2005,DCA30,Pham_dca_theory,Pham_trs,phamdinh2014,souza2016global,joki2017proximal,pang2017computing,lu2019enhanced,yin2015minimization}}. 

In the paper, we consider the standard DC program in form of
\begin{equation*}
\begin{split}
\text{min} &\{f(\textbf{x}) := f_1(\textbf{x}) - f_2(\textbf{x}):\textbf{x}\in\R^n\}, \\
\end{split}
\label{P}
\tag{$P$}
\end{equation*}
where $f_1$ and $f_2$ are proper closed and convex functions. Such a function $f$ is called a DC function, while $f_1-f_2$ is a DC decomposition of $f$ with $f_1$ and $f_2$ being DC components. Formulation \eqref{P} also includes convex constrained problems $\min\{\widetilde{f}_1(\textbf{x})-f_2(\textbf{x}):\textbf{x}\in X\}$, where $\widetilde{f}_1:\R^n\rightarrow\R$ is convex and $X\subseteq \R^n$ is nonempty, closed and convex. By introducing the indicator function of $X$, i.e., $\delta_X(\textbf{x}) = 0$ if $\textbf{x}\in X$ and $\delta_X(\textbf{x}) = \infty$ otherwise, this model recovers \eqref{P} as $\min\{(\widetilde{f}_1+\delta_{X})(\textbf{x})-f_2(\textbf{x}):\textbf{x}\in \R^n\}$. Throughout the manuscript, we make the next assumptions for problem (\ref{P}).
\begin{assumption}
	\label{intro:assump1}
	\begin{itemize}
		\item[(a)] $\dom (f_1) \subseteq Y \subseteq \text{int}(\dom(f_2))$, where $Y\subseteq\R^n$ is closed. 
		\item[(b)]  The function $f$ is bounded below, i.e., there exists a scalar $f^*$ such that {\normalfont $f(\textbf{x}) \geq f^*$} for all {\normalfont $\textbf{x}\in\R^n$}.
	\end{itemize}
\end{assumption}
The term (a) of Assumption \ref{intro:assump1} guarantees the nonemptiness of the subdifferential of $f_2$ at any point of $\dom (f_1)$. It will be also used in Theorem \ref{them:1}. 

The DC Algorithm (DCA) \cite{Lethi_2005,DCA30,Pham_dca_theory,Pham_trs} is renowned for solving \eqref{P}, which has been introduced by Pham Dinh Tao in 1985 and extensively developed by Le Thi Hoai An and Pham Dinh Tao since 1994. Specifically, at iteration $k$, DCA obtains the next iterate $\textbf{x}^{k+1}$ by solving the convex subproblem
\begin{equation*}
\textbf{x}^{k+1}\in \argmin\{f_1(\textbf{x})-\langle \textbf{y}^k, \textbf{x}\rangle:\textbf{x}\in\R^n\}, 
\end{equation*}
where $ \textbf{y}^k \in\partial f_2(\textbf{x}^k)$. It has been proved in \cite{Pham_dca_theory} that any limit point $\bar{\textbf{x}}$ of the generated sequence $\{\textbf{x}^k\}$ is a (DC) critical point of problem \eqref{P}, i.e., the necessary local optimality condition $\partial f_1(\bar{\textbf{x}})\cap \partial f_2(\bar{\textbf{x}}) \neq \emptyset$ is verified.

Recently, several accelerated algorithms for DC programming have been studied. Actacho et al. proposed
in \cite{BDCA_S,BDCA_NS} the boosted DC algorithms for unconstrained smooth/nonsmooth DC program, where the direction $\vx^{k+1}-\vx^k$, determined by the consecutive iterates of DCA, is verified as a descent direction of $f$ at $\textbf{x}^{k+1}$ when $f_2$ is strongly convex, thus a line-search procedure can be conducted along it to obtain a better candidate with lower objective value. Next, Niu et al. \cite{niu-port} developed a boosted DCA by incorporating the line-search procedure for DC program (both smooth and nonsmooth) with convex constraints. Later, Actacho et al. \cite{BDCA_NS_C} investigated the line-search idea in DC program with linear constraint, where a sufficient and necessary condition for determining the feasibility of  the direction $\vx^{k+1}-\vx^k$ is proposed. Different from the boosted approaches above, another type of acceleration is the momentum methods. There are two renowned momentum methods: the Polyak's heavy-ball method \cite{polyak1987introduction,zavriev1993heavy} and the Nesterov's acceleration technique \cite{nesterov1983method,nesterov2013gradient}. Indeed, Nesterov's acceleration belongs to the heavy-ball family. These two momentum methods have been successfully applied for various nonconvex optimizations problems (see, e.g., \cite{adly2021first,ochs2014ipiano,pock2016inertial,mukkamala2020convex}). In \cite{de2019inertial}, de Oliveira et al. proposed an inertial DC algorithm (InDCA) for problem \eqref{P} in the premise that $f_2$ is strongly convex, where the heavy-ball inertial-force procedure is introduced. \redtext{Phan et al. developed in \cite{nhat2018accelerated} an accelerated variant of DCA for \eqref{P} by incorporated the Nesterov's acceleration technique.} At the same time, Wen et al. \cite{wen2018proximal} used the same  technique for a special class of DC program, where $f_1$ is the sum of a smooth convex function and a (possibly non-smooth) convex function. Adly et al. \cite{adly2021first} proposed a gradient-based method, using both heavy-ball and Nesterov's inertial-forces, for optimizing a smooth function. They apply the algorithm for \eqref{P} by minimizing the DC Moreau envelope ${}^{\mu} f_1-{}^{\mu} f_2$\footnote{Let $\mu>0$ and $i=0,1$, ${}^{\mu}f_i$ is the Moreau envelope of $f_i$ with parameter $\mu$, defined as ${}^{\mu} f_i(\textbf{x}):= \min\{f_i(\textbf{y})+\frac{1}{2{\mu}}\Vert\textbf{x}-\textbf{y} \Vert^2:\textbf{y}\in\R^n\}.$}, which can return a critical point of $f_1-f_2$. However, their methods require computing the proximal operators of $f_1$ and $f_2$. The above inertial-type methods often help to improve the performances (compared with their counterparts without inertial-forces) in both the convergence speed and the solution quality. Concerning the optimality condition of the DC program \eqref{P}, the aforementioned methods only guarantee the critical point, which is weaker than the $d$-stationary point (see, e.g., \cite{pang2017computing,de2020abc}).  Next, we introduce some works about obtaining the $d$-stationary solution of \eqref{P}. By exploring the structure of $f_2$, i.e., $f_2$ is the supremum of finitely many convex smooth functions, Pang et al. \cite{pang2017computing} proposed a novel enhanced DC algorithm (EDCA) to obtain a $d$-stationary solution. Later, Lu et al. \cite{lu2019enhanced} developed an inexact variant of EDCA in order to make EDCA cheaper at each iteration. Furthermore, considering additionally the same structure of $f_1$ as that in \cite{wen2018proximal}, a proximal version of EDCA was developed, whose cost in each iteration is lower than the inexact version, besides, this algorithm incorporated the Nesterov's extrapolation for possible acceleration.

In this paper, based on the inertial DC algorithm (InDCA) \cite{de2019inertial}, we propose a refined version (RInDCA) equipped with larger inertial step-size. Indeed, both InDCA and RInDCA can handle exactness and inexactness in the solution of the convex subproblems. The basic idea of the exact versions of InDCA and RInDCA (denoted by \InDCAe\ and \RInDCAe) are described as follows. \RInDCAe\ obtains at the $k$-th iteration the next iterate $\textbf{x}^{k+1}$ by solving the convex subproblem
$$\textbf{x}^{k+1} \in \argmin \{f_1(\textbf{x}) - \langle \textbf{y}^k + \gamma(\textbf{x}^k-\textbf{x}^{k-1}),\textbf{x}\rangle:\textbf{x}\in\R^n\},$$
where $\textbf{y}^k \in \partial f_2(\textbf{x}^k)$.  
Our analysis shows that the inertial step-size $\gamma \in [0, (\sigma_1+\sigma_2)/2)$ is adequate for the convergence of \RInDCAe\ compared with $[0,\sigma_2/2)$ in \InDCAe, where $\sigma_1$ and $\sigma_2$ are the strong convex parameters of $f_1$ and $f_2$. Thus, the inertial range of \RInDCAe\ is larger than that of \InDCAe\ when $\sigma_1, \sigma_2 >0$. Empirically, larger step-size accelerates the convergence. In some applications, we will encounter the case where $f_1$ is strongly convex and $f_2$ is not, then \RInDCAe\ is applicable but not \InDCAe. Later, the details about the exact and inexact versions can be found in Section \ref{sec:3}. The benefit of larger inertial step-size is discussed in Section \ref{sec:6}.

Our \textit{contributions} are: (1) propose a refined inertial DC algorithm, which is equipped with larger inertial step-size compared with InDCA \cite{de2019inertial}. Besides, the relation between the inertial-type DCA and the classical DCA is pointed out. (2) establish the subsequential convergence of our refined version and the sequential convergence by further assuming the Kurdyka-Łojasiewicz (KL) property of the objective function.

The rest of the paper is organized as follows: Section \ref{sec:2} reviews some notations and preliminaries in convex and variational analysis. In Section \ref{sec:3}, we first introduce our refined inertial DC algorithm (RInDCA), followed by some DC programming applications. Next, the relation between the inertial-type DCA and the successive DCA is discussed. Section \ref{sec:4} focuses on establishing the subsequential convergence of RInDCA. By assuming additionally the KL property of the objective function, we prove in Section \ref{sec:5} the sequential convergence of RInDCA.  Section \ref{sec:6} summarizes the numerical results on checking copositivity of matrices and image denoising problem, in which the benefit of enlarged inertial step-size is demonstrated. Some concluding remarks are given in the final section.

\section{Notations and preliminaries}
\label{sec:2}
Let $\R^n$ denote the finite dimensional vector space endowed with the canonical inner product $\langle \cdot, \cdot \rangle$, and the Euclidean norm $\Vert \cdot \Vert$, i.e., $\Vert \cdot \Vert = \sqrt{\langle \cdot,\cdot\rangle}$. The entry of a vector \vx\ is denoted as $\vx_i$, and the entry of a matrix $\textbf{A}$ is denoted as $\textbf{A}_{i,j}$.

For an extended real-valued function $h:\R^n\rightarrow (-\infty,\infty]$, the set  $$\dom(h) := \{\textbf{x}\in \R^n:h(\textbf{x}) < \infty \}$$ denotes its effective domain. If $\dom (h) \neq \emptyset$, and $h$ does not attain the value $-\infty$, then $h$ is called a proper function. The set
$$\epi (h):=\{(\textbf{x},t):h(\textbf{x})\leq t,\textbf{x}\in \R^n,t\in \R\}$$ denotes the epigraph of $h$, and $h$ is closed (resp. convex) if $\epi (h)$ is closed (resp. convex). The conjugate function of $h$ is defined as 
\begin{equation*}
    h^*(\vy) = \sup\limits \{\langle \vy,\vx\rangle- h(\vx):\vx\in\R^n\}, \quad \vy\in \R^n.
\end{equation*}

Given a proper closed function $h: \R^n\rightarrow (-\infty,\infty]$, the Fr\'echet subdifferential of $h$ at $\textbf{x}\in \dom (h)$ is given by
\begin{equation*}
\partial^F h(\textbf{x}):= \left \{  \textbf{y}\in\R^n:\liminf\limits_{\textbf{z}\rightarrow \textbf{x}} \frac{h(\textbf{z})-h(\textbf{x})-\langle \textbf{y},\textbf{z}-\textbf{x}\rangle}{\Vert \textbf{z}-\textbf{x}\Vert} \geq 0\right\},
\end{equation*}
while for $\textbf{x}\notin \dom h$, $\partial^F h(\textbf{x}):= \emptyset$. The (limiting) subdifferential of $f$ at $\textbf{x}\in\dom (h)$ is defined as
\begin{align*}
\partial h(\textbf{x}) :=\{\textbf{y}\in\R^n:&\exists (\textbf{x}^k\rightarrow \textbf{x}, h(\textbf{x}^k)\rightarrow h(\textbf{x}), \textbf{y}^k \in \partial^F h(\textbf{x}^k)) \text{ such that } \textbf{y}^k\rightarrow \textbf{y}\},
\end{align*}
and $\partial h(\textbf{x}):=\emptyset$ if $\textbf{x}\notin \dom (h)$. Note that if $h$ is also convex, then the Fr\'echet subdifferential and the limiting subdifferential will coincide with the convex subdifferential, that is,  
$$\partial^{F} h(\textbf{x}) = \partial h(\textbf{x}) = \{\textbf{y}\in \R^n:h(\textbf{z}) \geq h(\textbf{x})+\langle \textbf{y},\textbf{z}-\textbf{x} \rangle \;\text{for all}\;\textbf{z}\in \R^n\}.$$

Given a nonempty set $\mathcal{C} \subseteq \R^n$, the distance from a point $\textbf{x}\in\R^n$ to $\mathcal{C}$ is denoted as $\text{dist}(\textbf{x},\mathcal{C}):=\inf\{\Vert \textbf{x}-\textbf{z}\Vert:\textbf{z}\in\mathcal{C}\}$. We now recall the Kurdyka-Łojasiewicz (KL) property \cite{bolte2007lojasiewicz,attouch2009convergence,attouch2013convergence}. For $\eta \in (0,\infty]$, we denote by $\Xi_{\eta}$ the set of all concave continuous functions $\varphi:[0,\eta)\rightarrow[0,\infty)$ that are continuously differentiable over $(0,\eta)$ with positive derivatives and satisfy $\varphi(0)=0$.
\begin{definition}[KL property]
	A proper closed function $h$ is said to satisfy the KL property at $\bar{\textbf{x}} \in \dom (\partial h) := \{\textbf{x}\in \R^n:\partial h(\textbf{x})\neq \emptyset\}$ if there exist $\eta \in (0,\infty]$, a neighborhood $U$ of $\bar{\textbf{x}}$, and a function $\varphi \in \Xi_{\eta}$ such that for all $\textbf{x}$ in the intersection
	\begin{equation*}
	U\cap \{\textbf{x}\in\R^n:h(\bar{\textbf{x}})<h(\textbf{x})<h(\bar{\textbf{x}})+\eta\},
	\end{equation*}
	it holds that 
	\begin{equation*}
	\varphi'(h(\textbf{x})-h(\bar{\textbf{x}}))\text{dist}(\zero,\partial h(\textbf{x})) \geq 1.
	\end{equation*}
	If $h$ satisfies the KL property at any point of $\dom (\partial h)$, then $h$ is called a KL function.
\end{definition}

The following uniformized KL property plays an important role in our sequential convergence analysis.
\begin{lemma}[Uniformized KL property, see \cite{bolte2014proximal}]
	\label{lem:UKL}
	Let $\Omega \subseteq \R^n$ be a compact set and let $h:\R^n \rightarrow (-\infty,\infty]$ be a proper closed function. If $h$ is constant on $\Omega$ and satisfies the KL property at each point of $\Omega$, then there exist $\varepsilon, \eta>0$ and $\varphi \in \Xi_{\eta}$ such that 
	\begin{equation*}
	{\normalfont\varphi'(h(\textbf{x})-h(\bar{\textbf{x}}))\text{dist}(\zero,\partial h(\textbf{x})) \geq 1}
	\end{equation*} 
	for any {\normalfont$\bar{\textbf{x}}\in \Omega$} and any {\normalfont$\textbf{x}$} satisfying {\normalfont$\text{dist}(\textbf{x},\Omega) < \varepsilon$ and $h(\bar{\textbf{x}}) < h(\textbf{x}) < h(\bar{\textbf{x}}) + \eta$}.
\end{lemma}

Now, let $h:\R^n\rightarrow (-\infty,\infty]$ be a proper closed and convex function, and let $\varepsilon > 0$, the set 
$$\partial_{\varepsilon} h(\textbf{x}) := \{\textbf{y}\in \R^n:h(\textbf{z}) \geq h(\textbf{x})+\langle \textbf{y},\textbf{z}-\textbf{x} \rangle -\varepsilon \;\text{for all}\;\textbf{z}\in \R^n\}$$ denotes the $\varepsilon$-subdifferential of $h$ at $\textbf{x}$, and any point in $\partial_{\varepsilon} h(\textbf{x})$ is called a $\varepsilon$-subgradient of $h$ at $\textbf{x}$. Clearly, $\partial_{\varepsilon} h(\textbf{x}) \neq \emptyset$ implies that
$\textbf{x} \in \dom (h)$. 
 A proper closed function $h$ is called $\sigma$-strongly convex (cf. $\sigma$-convex) with $\sigma \geq 0$ if for any $\textbf{x}, \textbf{z} \in \dom (h)$ and $\lambda \in [0,1]$, it holds that
\[h(\lambda \textbf{x} +(1-\lambda)\textbf{z}) \leq \lambda h(\textbf{x}) + (1-\lambda)h(\textbf{z})-\frac{\sigma}{2}\lambda(1-\lambda)\Vert \textbf{x}-\textbf{z}\Vert^2.\] 
Moreover, let $h$ be a $\sigma$-convex function. Then it was known \cite{Beck_1order} that for any $\textbf{x}\in \dom (\partial h)$, $\textbf{y} \in \partial h(\textbf{x})$ and $\textbf{z}\in \dom (h)$, we have
		\[h(\textbf{z})\geq h(\textbf{x})+\langle \textbf{y},\textbf{z}-\textbf{x}\rangle+\frac{\sigma}{2}\Vert \textbf{z}-\textbf{x}\Vert^2.\] 
Finally, we give a result related to strong convexity and $\varepsilon$-subdifferential, which will be used in analyzing our refined inexact DC algorithm. 
		
\begin{lemma}[\cite{partialDC}]
	\label{lem:2.1}
	Let $h:\R^n \rightarrow (-\infty,\infty]$ be a $\sigma$-convex function, and let $\varepsilon \geq 0$, $t \in (0,1]$.  Then for any {\normalfont$\textbf{x}\in \dom (h)$}, {\normalfont$\textbf{z}\in\dom (h)$}, and {\normalfont$\textbf{y}\in\partial_{\varepsilon} h(\textbf{x})$},
		{\normalfont
		\begin{equation}
		\label{eq:00}
		    h(\textbf{z})\geq h(\textbf{x}) +\langle \textbf{y},\textbf{z}-\textbf{x}\rangle + \frac{\sigma(1-t)}{2}\Vert \textbf{z}-\textbf{x}\Vert^2 - \frac{\varepsilon}{t}.
		\end{equation}
		}
\end{lemma}

\section{A refined inertial DC algorithm} \label{sec:3}
In this section, we first describe a general inertial-type DC algorithm for problem \eqref{P}, which is taken from \cite{de2019inertial} except that the values of $\lambda$ and $\gamma$ are not specified. By setting different  $\lambda$ and $\gamma$ we derive the classical DCA, the inertial DC algorithm (InDCA) proposed in \cite{de2019inertial}, and our refined version (RInDCA). Next, some practical applications of DC programming problems are introduced, followed by a discussion of the relation between the inertial-type DCA and the successive DCA.

From now on, suppose that the objective function $f$ has a DC decomposition $f = f_1-f_2$, where $f_1$ is $\sigma_1$-convex and $f_2$ is $\sigma_2$-convex. We describe below in Algorithm \ref{algo:RInDCA} a general inertial-type DC algorithm for \eqref{P}.
\begin{algorithm}[ht!]
	\caption{A general inertial-type DC algorithm for \eqref{P}}
	\label{algo:RInDCA}
	\KwIn{$\textbf{x}^0 \in \dom (f_1)$; $\lambda \in[0,1)$; $\gamma \geq 0$; $\textbf{x}^{-1} = \textbf{x}^0$.}
	
	\For{$k=0, 1, 2, \cdots$}
	{   
		find $\textbf{x}^{k+1}\in \R^n$ such that
		\begin{equation}
		\label{eq:3.2}
		\partial_{\varepsilon_{k+1}} f_1(\textbf{x}^{k+1}) \cap (\partial f_2(\textbf{x}^k) + \gamma(\textbf{x}^k-\textbf{x}^{k-1}))\neq \emptyset  
		\end{equation}
		with $0 \leq \varepsilon_{k+1} \leq \lambda \frac{\sigma_2}{2}\Vert \textbf{x}^{k+1}-\textbf{x}^k\Vert^2$.
		
	}	
\end{algorithm}

Note that Algorithm \ref{algo:RInDCA} includes exact and inexact cases. The exact case occurs when $\lambda = 0$, then $\varepsilon_{k+1} = 0$ and $\partial_{\varepsilon_{k+1}} f_1(\textbf{x}^{k+1})  = \partial f_1(\textbf{x}^{k+1})$, thus the iteration point $\textbf{x}^{k+1}$ in \eqref{eq:3.2} is obtained by exactly solving the next convex subproblem:
\begin{equation}
    \label{eq:InDCA}
    \textbf{x}^{k+1} \in \argmin \{f_1(\textbf{x}) - \langle \textbf{y}^k + \gamma(\textbf{x}^k-\textbf{x}^{k-1}),\textbf{x}\rangle:\textbf{x}\in\R^n\},
\end{equation}
where $\textbf{y}^k\in\partial f_2(\textbf{x}^k)$. On the other hand, the inexact case occurs when $\lambda > 0$ and $\sigma_2 > 0$, then the iterate $\vx^{k+1}$ in \eqref{eq:3.2}, which is an approximate minimizer of \eqref{eq:InDCA}, could be computed by some bundle methods (see \cite{de2019inertial}). The classical DCA is derived when setting $\lambda = 0$ and $\gamma = 0$ in Algorithm \ref{algo:RInDCA}.  Next, we denote the exact (resp. inexact) versions of InDCA and RInDCA as \InDCAe\ and \RInDCAe\ (resp. \InDCAine\ and \RInDCAine). The following table summarizes the four algorithms which are recovered by specifying some parameters in Algorithm \ref{algo:RInDCA} together with the assumptions about $\sigma_1$ and  $\sigma_2$. 
\begin{table}[h]
	\begin{center}
		\begin{tabular}{|c||r|c|c|c|}
	\hline
	Algorithm & $\lambda$	& $\gamma$ & t & Assumption \\
	\hline\hline
	\InDCAe & 0 & $[0,\sigma_2/2)$ & - &$\sigma_2>0$\\
	\InDCAine & (0,1) & $[0,(1-\lambda)\sigma_2/2)$ & -&$\sigma_2>0$\\
	\RInDCAe & 0 & $[0,(\sigma_1+\sigma_2)/2) $ & - &$\sigma_1+\sigma_2>0$ \\
	\RInDCAine & (0,1) & $[0, (\sigma_1(1-t)+\sigma_2-\lambda\sigma_2/t)/2)$ & (0,1] &$\sigma_2>0$\\
	\hline
\end{tabular}
	\end{center}
\end{table}

Comparing \InDCAe\ with \RInDCAe, the inertial range of the latter is larger than the former if $\sigma_1>0$. Moreover, \RInDCAe\ is applicable when $\sigma_2=0$ and $\sigma_1>0$, while this is not the case for \InDCAe. Next, for comparing \InDCAine\ and \RInDCAine, we need to optimize $t$ in \RInDCAine\ such that the supremum of the inertial step-size $\gamma$ is maximized. To this end, fixing $\lambda\in(0,1)$ and let $\sigma(\lambda,t):=\sigma_1(1-t)+\sigma_2-\lambda\sigma_2/t$, we can maximize  $\sigma(\lambda,t)$ with respect to $t\in[0,1)$. It is easy to verify that
\begin{equation*}
 \argmax\{\sigma(\lambda,t):t\in(0,1]\} =
\begin{cases}
1 & \text{if $\sigma_1 = 0$},\\
1 & \text{if $\sigma_1 \neq 0, \sqrt{\lambda\sigma_2/\sigma_1} \geq 1$,}\\
\sqrt{\lambda\sigma_2/\sigma_1}  & \text{if $\sigma_1 \neq 0, \sqrt{\lambda\sigma_2/\sigma_1} < 1$.}
\end{cases}
\end{equation*}
For the first two cases, $\sigma(\lambda,1) = (1-\lambda)\sigma_2$, this implies the same inertial range between \InDCAine\ and \RInDCAine. Next, for the last case, we have $\sigma(\lambda,1)<\sigma(\lambda,\sqrt{\lambda\sigma_2/\sigma_1})$, thus the inertial range of \RInDCAine\ is larger than that in \InDCAine. Now, let $H_1(\lambda) := \sigma(\lambda,\sqrt{\lambda\sigma_2/\sigma_1})/2$ and $H_2(\lambda) :=(1-\lambda)\sigma_2/2$, we visualize in Fig. \ref{Inertial:fig1} the values of $H_1$ for $(\sigma_1,\sigma_2) \in\{ (1,1), (2,1), (3,1), (4,1)\}$ as well as the values of $H_2$ for $\sigma_1 = 1$ with respect to $\lambda\in(0,1)$ respectively.
\begin{figure}[h!]
	\centering
	\includegraphics[width=0.6\linewidth]{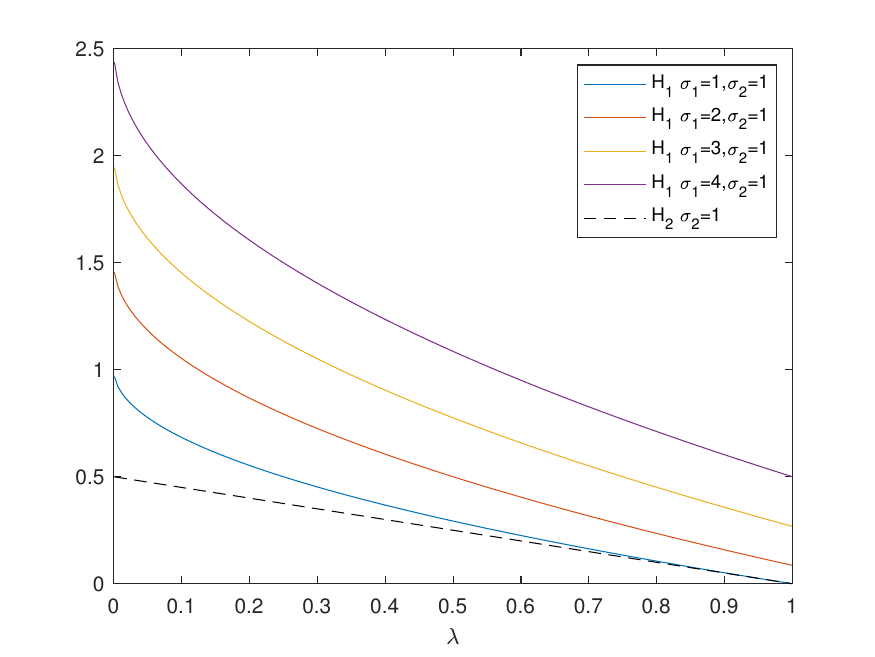}
	\caption{The values of $H_1$ with $(\sigma_1,\sigma_2) \in\{ (1,1), (2,1), (3,1), (4,1)\}$  and $H_2$ with respect to $\lambda \in (0,1)$}
	\label{Inertial:fig1}
\end{figure}
It can be observed that when fixing $\lambda\in (0,1)$ and $\sigma_2=1$, the value of $H_1(\lambda)$ is getting larger as $\sigma_1$ increases, moreover, all of these values are larger than $H_2(\lambda)$. This observation implies that a larger strong convexity parameter of $f_1$ will result in a larger inertial range.

Now, we introduce some DC programming problems, in which the first DC components are strongly convex.
\begin{example}[Convex constraint quadratic problem]
\label{exam:3.1}
    Consider the quadratic problem with convex constraint 
    $$\min \{f(\textbf{x}) =  \frac{1}{2}\textbf{x}^{\top} \textbf{A}\textbf{x}+\delta_Y(\textbf{x}):\textbf{x}\in\R^n\},$$
    where $\textbf{A}\in\R^{n\times n}$ is symmetric and $Y$ is a closed convex set.  This model has a DC decomposition $f= f_1-f_2$ with $f_1(\textbf{x})=\frac{L}{2}\Vert \textbf{x}\Vert^2 + \delta_{Y}(\textbf{x})$ and $f_2(\textbf{x}) = \frac{L}{2}\Vert \textbf{x}\Vert^2 - \frac{1}{2}\textbf{x}^{\top}\textbf{A}\textbf{x}$, where $L \geq \Vert \textbf{A}\Vert$\footnote{$\Vert \textbf{A}\Vert$ is the spectral norm of $\textbf{A}$.}. When $Y$ is the nonnegative orthant $\R_+^n$, DCA applied for the above DC decomposition
    has been used for checking copositivity of matrices \cite{JP_copo}. Note that for this DC decomposition, the first DC component is $L$-convex.  
\end{example}
 
\begin{example}[1D signals denoising]
	\label{exam:3.2}
	Consider the nonconvex sparse signal recovery problem for 1D signals:
	\begin{equation*}
	\min \{\frac{\mu}{2}\Vert \textbf{x}-\textbf{b}\Vert^2 + \sum_{i=1}^{n-1}\phi(|\textbf{x}_{i+1}-\textbf{x}_i|):\vx\in\R^n\},
	\end{equation*}
	where $\mu >0$ is a trade-off constant, $\textbf{b}$ is the noise signal in $\R^n$, and $\phi$ is a concave function for inducing sparsity, e.g., $\phi(r) := \log(1+2r)/r$. In this model, taking $\phi(r) = \log(1+2r)/r$, we obtain a DC decomposition as $f_1 - f_2$ with $f_1(\textbf{x}) = \frac{\mu}{2}\Vert \textbf{x}-\textbf{b}\Vert^2 + \sum_{i=1}^{n-1}|\textbf{x}_{i+1}-\textbf{x}_i|$ and $f_2 = \sum_{i=1}^{n-1}|\textbf{x}_{i+1}-\textbf{x}_i| - \sum_{i=1}^{n-1}\phi(|\textbf{x}_{i+1}-\textbf{x}_i|)$. Note that $f_1$ is strongly convex, while $f_2$ is not. Thus InDCA can not be directly applied. In order to apply InDCA, a strong convexity term $\psi:\textbf{x}\mapsto \rho\Vert \textbf{x} \Vert^2/2$ is added into both $f_1$ and $f_2$ in \cite{de2019inertial}, yielding the new DC decomposition $(f_1+\psi) - (f_2+\psi)$, where the DC components are both strongly convex.
\end{example}

\begin{remark}
As shown in Examples \ref{exam:3.1} and \ref{exam:3.2}, the first DC components are strongly convex. However, InDCA only takes into account the strong convexity of the second DC component. This often causes the inertial range smaller than that of our refined version RInDCA, which involves the strong convexity of both DC components. 
\end{remark}

Before ending this section, we mention out that the inertial-type DCA (both \InDCAe\ and \RInDCAe) is related to the successive DCA, namely SDCA (see \cite{DCA30}). Let us take at the $k$-th iteration of DCA the successive DC decomposition $f= (f_1+\varphi^k) - (f_2+\varphi^k)$ with $\varphi^k(\vx) = \frac{\gamma}{2}\|\vx - \vx^{k-1}\|^2$, then DCA applied to this special DC decomposition yields
\begin{equation}
\label{eq:SDCA}
    \vx^{k+1}\in \argmin \{ f_1(\vx) + \varphi^k(\vx) - \langle \nabla \varphi^k(\vx^k) + \vy^k , \vx \rangle:\vx\in\R^n\},
\end{equation}
where $\vy^k\in \partial f_2(\vx^k)$. Then \eqref{eq:SDCA} reads as
\begin{equation}
    \label{eq:DCAphik}
    \vx^{k+1}\in \argmin \{ f_1(\vx) + \varphi^k(\vx) - \langle \vy^k +\gamma (\vx^k-\vx^{k-1}) , \vx \rangle :\vx\in\R^n\}.
\end{equation}
Comparing SDCA described in \eqref{eq:DCAphik} with \RInDCAe\ (or \InDCAe) in \eqref{eq:3.2}, we observe that the only difference is the presence of the proximal term $\varphi^k(\vx)$ in \eqref{eq:DCAphik}. This term plays a role of regularizer which can be interpreted as finding a point $\vx^{k+1}$ close to $\vx^{k-1}$ and a minimizer of the function $\textbf{x}\mapsto f_1(\textbf{x}) - \langle \textbf{y}^k + \gamma(\textbf{x}^k-\textbf{x}^{k-1}),\textbf{x}\rangle$. This is the basic idea of proximal point methods, thus this SDCA is often called proximal DCA as well. The absence of $\varphi^k$ in \RInDCAe\ means that the iteration point $\vx^{k+1}$ is not supposed to be close to $\vx^{k-1}$, which will lead to potential acceleration. In conclusion, the term $\varphi^k$ plays a key role to make \RInDCAe\ differ from SDCA, and possibly yields acceleration in \RInDCAe.  

\begin{remark}
Indeed, the convex problem of \eqref{eq:DCAphik} is irrelevant to $\vx^{k-1}$ since
\begin{equation*}
    \eqref{eq:DCAphik} = \argmin \{ f_1(\vx) + \frac{\gamma}{2}\|\vx\|^2 - \langle \vy^k +\gamma \vx^k , \vx \rangle :\vx\in\R^n\},
\end{equation*}
which is exactly the subproblem of DCA for the fixed DC decomposition $$f = (f_1 + \frac{\gamma}{2}\Vert \cdot \Vert^2) - (f_2 + \frac{\gamma}{2}\Vert \cdot \Vert^2).$$
\end{remark}

\section{Subsequential convergence analysis}
\label{sec:4}
This section focuses on  showing the  subsequential convergence of RInDCA (both \RInDCAe\ and \RInDCAine), some results for further developing the sequential convergence of RInDCA in the next section are also proved.  
\subsection{Subsequential convergence of \RInDCAe}
The following assumption is made in this subsection.
\begin{assumption}\label{assump:RInDCAe}
    Suppose that the DC components $f_1$ and $f_2$ are respectively $\sigma_1$-convex and $\sigma_2$-convex with $\sigma_1+\sigma_2>0$.
\end{assumption} 
Now, we start to establish the subsequential convergence of \RInDCAe. 
\begin{lemma}
	\label{prop:3.1}
	Let {\normalfont$\{\textbf{x}^k\}$} be the sequence generated by \RInDCAe\ and $\gamma \in [0,(\sigma_1+\sigma_2)/2)$, then the sequence {\normalfont$\{f(\textbf{x}^k)+ \frac{\sigma_1+\sigma_2-\gamma}{2}\Vert \textbf{x}^k-\textbf{x}^{k-1}\Vert^2 \}$} is nonincreasing and for all $k \geq 0$,
	{\normalfont
	\begin{align}
	\label{eq:prop:3.1}
	f(\textbf{x}^{k+1}) + \frac{\sigma_1+\sigma_2-\gamma}{2}\Vert \textbf{x}^{k+1}-\textbf{x}^k\Vert^2 
	\leq& f(\textbf{x}^{k}) + \frac{\sigma_1+\sigma_2-\gamma}{2}\Vert \textbf{x}^{k}-\textbf{x}^{k-1}\Vert^2 \nonumber\\
	&-\frac{\sigma_1+\sigma_2-2\gamma}{2}\Vert \textbf{x}^k-\textbf{x}^{k-1}\Vert^2.
	\end{align}}
\end{lemma}
\begin{proof}
	Because $\textbf{x}^{k+1}$ satisfies (\ref{eq:InDCA}), thus $\textbf{y}^k+\gamma(\textbf{x}^k-\textbf{x}^{k-1}) \in \partial f_1(\textbf{x}^{k+1})$, then by $\sigma_1$-convexity of $f_1$, we have 
	\begin{equation}
	\label{eq:2.6}
	f_1(\textbf{x}^k) \geq f_1(\textbf{x}^{k+1}) + \langle \textbf{y}^k + \gamma(\textbf{x}^k-\textbf{x}^{k-1}),\textbf{x}^k-\textbf{x}^{k+1}\rangle+\frac{\sigma_1}{2}\Vert \textbf{x}^{k+1}-\textbf{x}^k\Vert^2.
	\end{equation} 
	On the other hand, it follows from $\textbf{y}^k \in \partial f_2(\textbf{x}^k)$ and $\sigma_2$-convexity of $f_2$ that 
	\begin{equation}
	\label{eq:2.7}
	f_2(\textbf{x}^{k+1}) \geq f_2(\textbf{x}^k) + \langle \textbf{y}^k,\textbf{x}^{k+1}-\textbf{x}^k \rangle +\frac{\sigma_2}{2}\Vert \textbf{x}^{k+1}-\textbf{x}^k\Vert^2.
	\end{equation}
	Summing (\ref{eq:2.6}) to (\ref{eq:2.7}) and reshuffling the terms, we derive that 
	\begin{equation}
	\label{eq:2.8}
	f(\textbf{x}^k) \geq f(\textbf{x}^{k+1}) + \gamma\langle \textbf{x}^k-\textbf{x}^{k-1},\textbf{x}^k-\textbf{x}^{k+1}\rangle + \frac{\sigma_1+\sigma_2}{2}\Vert \textbf{x}^{k+1}-\textbf{x}^k\Vert^2.
	\end{equation}
	By applying $\langle \textbf{x}^k-\textbf{x}^{k-1},\textbf{x}^k-\textbf{x}^{k+1}\rangle \geq -(\Vert \textbf{x}^k-\textbf{x}^{k-1}\Vert^2 + \Vert \textbf{x}^k-\textbf{x}^{k+1}\Vert^2)/{2}$ to \eqref{eq:2.8}, we obtain
	\begin{align*}
	f(\textbf{x}^{k+1}) + \frac{\sigma_1+\sigma_2-\gamma}{2}\Vert \textbf{x}^{k+1}-\textbf{x}^k\Vert^2 
	\leq& f(\textbf{x}^{k}) + \frac{\sigma_1+\sigma_2-\gamma}{2}\Vert \textbf{x}^{k}-\textbf{x}^{k-1}\Vert^2 \nonumber\\
	&-\frac{\sigma_1+\sigma_2-2\gamma}{2}\Vert \textbf{x}^k-\textbf{x}^{k-1}\Vert^2.
	\end{align*}
	Moreover, we get from $\gamma < (\sigma_1+\sigma_2)/2$ that $(\sigma_1 + \sigma_2 -2\gamma)/2 > 0$, thus the sequence $\{f(\textbf{x}^k)+ \frac{\sigma_1+\sigma_2-\gamma}{2}\Vert \textbf{x}^k-\textbf{x}^{k-1}\Vert^2 \}$ is nonincreasing. 
	\qed
\end{proof}

\begin{theorem}
	\label{them:1}
	Let {\normalfont$\{\textbf{x}^k\}$} be the sequence generated by \RInDCAe\ and $\gamma \in [0,(\sigma_1+\sigma_2)/2)$. Then the following statements hold:
	\begin{enumerate}
		\item[(i)] {\normalfont$\lim\limits_{k\rightarrow \infty}{\Vert \textbf{x}^{k}-\textbf{x}^{k-1}\Vert  = 0}$};
		\item[(ii)] If {\normalfont$\{\textbf{x}^k\}$} is bounded, then any limit point of {\normalfont$\{\textbf{x}^k\}$} is a critical point of \eqref{P}. 
	\end{enumerate}
\end{theorem}

\begin{proof}
	$(i)$ Let us denote $\eta_1 := (\sigma_1+\sigma_2-\gamma)/2$ and $\eta_2 := (\sigma_1+\sigma_2-2\gamma)/{2}$. Clearly, $\eta_1>0$, $\eta_2 >0$, and the inequality \eqref{eq:prop:3.1} reads as 
	\begin{equation}
	\label{eq1:them3.1}
	f(\textbf{x}^{k+1}) + \eta_1\Vert \textbf{x}^{k+1}-\textbf{x}^k\Vert^2 \leq f(\textbf{x}^{k}) + \eta_1\Vert \textbf{x}^{k}-\textbf{x}^{k-1}\Vert^2 -\eta_2 \Vert \textbf{x}^k-\textbf{x}^{k-1}\Vert^2.
	\end{equation}
	By summing (\ref{eq1:them3.1}) from $k=0$ to $n$, we obtain that 
	\begin{equation}\label{eq:thm3_1}
	\eta_2\sum_{k=0}^{n}\Vert \textbf{x}^{k}-\textbf{x}^{k-1}\Vert^2 \leq f(\textbf{x}^0) - \left( f(\textbf{x}^{n+1})+\eta_1\Vert \textbf{x}^{n+1}-\textbf{x}^n\Vert^2\right). 
	\end{equation}
	Recalling that $f$ is bounded below by a scalar $f^*$ (term (b) of Assumption \ref{intro:assump1}), then we have
	$$f(\textbf{x}^{n+1})+\eta_1\Vert \textbf{x}^{n+1}-\textbf{x}^n\Vert^2\geq f^*,$$
	thus we obtain
	$$\sum_{k=0}^{n}\Vert \textbf{x}^{k}-\textbf{x}^{k-1}\Vert^2 \leq \frac{f(\textbf{x}^0)-f^*}{\eta_2},\quad \forall n\in \N.$$ 
	Thus, $\sum_{k=0}^{\infty} \Vert \textbf{x}^k-\textbf{x}^{k-1}\Vert^2 < \infty$ which implies $\lim\limits_{k\rightarrow \infty} \Vert \textbf{x}^k -\textbf{x}^{k-1}\Vert = 0.$\\
	$(ii)$ The boundedness of $\{\textbf{x}^k\}$ together with term (a) of Assumption \ref{intro:assump1} implies the boundedness of $\{\textbf{y}^k\}$ (see Theorem 3.16 in \cite{Beck_1order}). Let $\bar{\textbf{x}}$ be any limit point of $\{\textbf{x}^k\}$, then there exists a convergent subsequence such that $\lim\limits_{i\rightarrow \infty}\textbf{x}^{k_i} = \bar{\textbf{x}}$.  Moreover, since $\{\textbf{y}^{k_i}\}$ is bounded in $\R^n$, there exists a convergent subsequence of $\{\textbf{y}^{k_i}\}$. Without loss of generality, we can assume that the sequence $\{\textbf{y}^{k_i}\}$ is convergent and $\lim\limits_{i\rightarrow \infty}\textbf{y}^{k_i} = \bar{\textbf{y}}$. Then taking into account that $\textbf{y}^{k_i}\in \partial f_2(\textbf{x}^{k_i})$ for all $i$, $\textbf{y}^{k_i} \rightarrow \bar{\textbf{y}}$,  $\textbf{x}^{k_i} \rightarrow \bar{\textbf{x}}$, and the closedness of the graph $\partial f_2$ (see Theorem 24.4 in \cite{rockafellar1970convex}), we obtain 
	\begin{equation}
	    \label{eq:thm3_3}
	    \bar{\textbf{y}}\in \partial f_2(\bar{\textbf{x}}).
	\end{equation}
	On the other hand, the next relation holds for all $i$
	\[\textbf{y}^{k_i}+\gamma(\textbf{x}^{k_i}-\textbf{x}^{k_i-1})\in\partial f_1(\textbf{x}^{k_i+1}). \]  
	By the fact that $\lim\limits_{i\rightarrow\infty}\textbf{x}^{k_i+1} = \bar{\textbf{x}}$ since $\lim\limits_{i\rightarrow \infty}{\Vert \textbf{x}^{k_i+1}-\textbf{x}^{k_i}\Vert  = 0}$, then  combining $\textbf{y}^{k_i} \rightarrow \bar{\textbf{y}}$, $\gamma (\textbf{x}^{k_i}-\textbf{x}^{k_i-1}) \rightarrow 0$, $\textbf{x}^{k_i+1} \rightarrow \bar{\textbf{x}}$, and the closedness of the graph $\partial f_1$, we have \begin{equation}
	    \label{eq:thm3_4}
	    \bar{\textbf{y}}\in \partial f_1(\bar{\textbf{x}}).
	\end{equation}
	We conclude from \eqref{eq:thm3_3} and \eqref{eq:thm3_4} that $\bar{\textbf{y}} \in \partial f_1(\bar{\textbf{x}})\cap \partial f_2(\bar{\textbf{x}}) \neq \emptyset$, thus $\bar{\textbf{x}}$ is a critical point of problem (\ref{P}).
	\qed
\end{proof}

\subsection{Subsequential convergence of \RInDCAine}
\label{sub:3.2}
The following assumption is made in this part.
\begin{assumption}\label{assump:RInDCAine}
    Suppose that the second DC component $f_2$ is $\sigma_2$-convex with $\sigma_2>0$.
\end{assumption} 
Now, we start to establish the subsequential convergence of \RInDCAine. 

\begin{lemma}
\label{lem:3.2}
	Let {\normalfont$\{\textbf{x}^k\}$} be the sequence generated by \RInDCAine\ and $\gamma \in [0,\bar{\sigma}/2)$, where $\bar{\sigma} = \sigma_1(1-t) + \sigma_2 -\lambda \sigma_2/t$ with $\lambda\in(0,1)$, $t\in(0,1]$. Then the sequence {\normalfont$\{f(\textbf{x}^k) + \frac{\bar{\sigma}-\gamma}{2}\Vert \textbf{x}^k-\textbf{x}^{k-1}\Vert^2\}$} is nonincreasing and for all $k\geq 0$, we have 
	\begin{align}
	\label{eq:prop3.2.1}
	f({\normalfont \textbf{x}^{k+1} }) + \frac{\bar{\sigma} -\gamma}{2}\Vert {\normalfont \textbf{x}^{k+1}-\textbf{x}^k}\Vert^2 
	\leq& f({\normalfont \textbf{x}^k }) + \frac{\bar{\sigma} -\gamma}{2}\Vert {\normalfont \textbf{x}^k-\textbf{x}^{k-1} }\Vert^2 \nonumber\\
	&-\frac{\bar{\sigma}-2\gamma}{2}\Vert {\normalfont\textbf{x}^k-\textbf{x}^{k-1} } \Vert^2. 
	\end{align}
\end{lemma}

\begin{proof}
	Because $\textbf{x}^{k+1}$ satisfies (\ref{eq:3.2}), let $\textbf{y}^k \in \partial f_2(\textbf{x}^k)$ such that $\textbf{y}^k+\gamma(\textbf{x}^k-\textbf{x}^{k-1}) \in \partial_{\varepsilon^{k+1}} f_1(\textbf{x}^{k+1})$, we have
	\begin{align}
	\label{eq:9}
	f_1(\textbf{x}^k) \overset{\eqref{eq:00}}{\geq}& f_1(\textbf{x}^{k+1}) + \langle \textbf{y}^k + \gamma(\textbf{x}^k-\textbf{x}^{k-1}),\textbf{x}^k-\textbf{x}^{k+1}\rangle \nonumber\\
	&+\frac{\sigma_1(1-t)}{2}\Vert \textbf{x}^{k+1}-\textbf{x}^k\Vert^2 -\frac{\varepsilon^{k+1}}{t} \nonumber\\
	\geq& f_1(\textbf{x}^{k+1}) + \langle \textbf{y}^k + \gamma(\textbf{x}^k-\textbf{x}^{k-1}),\textbf{x}^k-\textbf{x}^{k+1}\rangle \\
	&+\frac{\sigma_1(1-t)-\lambda   \sigma_2/t}{2}\Vert \textbf{x}^{k+1}-\textbf{x}^k\Vert^2, \nonumber
	\end{align}
	where the second inequality is implied by the fact  $\varepsilon^{k+1}\leq \lambda\frac{\sigma_2}{2}\Vert \textbf{x}^{k+1}-\textbf{x}^k\Vert^2$. 
	On the other hand, $\textbf{y}^k \in \partial f_2(\textbf{x}^k)$ and $\sigma_2$-convexity of $f_2$ imply that 
	\begin{equation}
	\label{eq:10}
	f_2(\textbf{x}^{k+1}) \geq f_2(\textbf{x}^k) + \langle \textbf{y}^k,\textbf{x}^{k+1}-\textbf{x}^k \rangle +\frac{\sigma_2}{2}\Vert \textbf{x}^{k+1}-\textbf{x}^k\Vert^2.
	\end{equation}
	Summing (\ref{eq:9}) to (\ref{eq:10}) and reshuffling the terms, we obtain that 
	\begin{equation}
	\label{eq:11}
	f(\textbf{x}^k) \geq f(\textbf{x}^{k+1}) + \gamma\langle \textbf{x}^k-\textbf{x}^{k-1},\textbf{x}^k-\textbf{x}^{k+1}\rangle + \frac{\bar{\sigma}}{2}\Vert \textbf{x}^{k+1}-\textbf{x}^k\Vert^2.
	\end{equation}
	By applying $\langle \textbf{x}^k-\textbf{x}^{k-1},\textbf{x}^k-\textbf{x}^{k+1}\rangle \geq -(\Vert \textbf{x}^k-\textbf{x}^{k-1}\Vert^2 + \Vert \textbf{x}^k-\textbf{x}^{k+1}\Vert^2)/{2}$ to (\ref{eq:11}), we obtain (\ref{eq:prop3.2.1}). Moreover, $\gamma \in [0,\bar{\sigma}/2)$ implies $(\bar{\sigma} -2\gamma)/2> 0$, thus the sequence {\normalfont$\{f(\textbf{x}^k) + \frac{\bar{\sigma} -\gamma}{2}\Vert \textbf{x}^k-\textbf{x}^{k-1}\Vert^2\}$} is nonincreasing. \qed
\end{proof}

The proof of the next theorem is omitted since it follows from the similar arguments as in Theorem \ref{them:1}.
\begin{theorem}
	\label{them:2}
	Let {\normalfont$\{\textbf{x}^k\}$} be the sequence generated by \RInDCAine\ and $\gamma \in [0,\bar{\sigma}/2)$, where $\bar{\sigma} = \sigma_1(1-t) + \sigma_2 -\lambda \sigma_2/t$ with $\lambda\in(0,1)$, $t\in(0,1]$. Then the following statements hold:
	\begin{enumerate}
		\item[(i)] {\normalfont$\lim\limits_{k\rightarrow \infty}{\Vert \textbf{x}^k-\textbf{x}^{k-1}\Vert = 0}$};
		\item[(ii)] If {\normalfont$\{\textbf{x}^k\}$} is bounded, then any limit point of {\normalfont$\{\textbf{x}^k\}$} is a critical point of (\ref{P}). 
	\end{enumerate}
\end{theorem}

\subsection{Some results on the sequences obtained by RInDCA}
In this subsection, we provide some important results on the sequences obtained by \RInDCAe\ and \RInDCAine.
\begin{proposition}
	\label{prop:3.3}
	Suppose that the assumptions in Theorem \ref{them:1} hold. Let {\normalfont$\{\textbf{x}^k\}$} be the sequence generated by \RInDCAe\ and let $\Omega_1$ be the set of the limit points of {\normalfont$\{\textbf{x}^k\}$}. If {\normalfont$\{\textbf{x}^k\}$} is bounded, then we have
	\begin{itemize}
		\item[(i)] {\normalfont$\lim\limits_{k\rightarrow \infty} f(\textbf{x}^k):= \zeta$} exists;
		\item[(ii)]  $f \equiv \zeta$ on $\Omega_1$.
	\end{itemize}
\end{proposition}

\begin{proof}
	$(i)$ Lemma \ref{prop:3.1} implies that the sequence $\{f(\textbf{x}^k)+ \frac{\sigma_1+\sigma_2-\gamma}{2}\Vert \textbf{x}^k-\textbf{x}^{k-1}\Vert^2 \}$ is nonincreasing, combining the fact that this sequence is bounded below, thus it converges. On the other hand, (ii) of Theorem \ref{them:1} implies that $\lim\limits_{k\rightarrow \infty} \Vert \textbf{x}^k-\textbf{x}^{k-1}\Vert^2 = 0$. Therefore, $\lim\limits_{k\rightarrow \infty} f(\textbf{x}^k):=\zeta$ exists. \\
	$(ii)$ Because $\Omega_1$ is nonempty, thus given any $\bar{\textbf{x}}\in \Omega_1$, there exists a subsequence $\{\textbf{x}^{k_i}\}$ such that $\lim\limits_{i\rightarrow \infty} \textbf{x}^{k_i} = \bar{\textbf{x}}$. Then, we get from $\textbf{y}^{k_i} \in \partial f_2(\textbf{x}^{k_i})$ and $\sigma_2$-convexity of $f_2$ that 
	\begin{equation}
	\label{eq:001}
	    f_2(\textbf{x}^{k_i+1}) \geq f_2(\textbf{x}^{k_i}) + \langle \textbf{y}^{k_i} , \textbf{x}^{k_i+1}-\textbf{x}^{k_i} \rangle + \frac{\sigma_2}{2}\Vert \textbf{x}^{k_i+1}-\textbf{x}^{k_i}\Vert^2.
	\end{equation}
	Similarly, we obtain as well from $\textbf{y}^{k_i} + \gamma(\textbf{x}^{k_i}-\textbf{x}^{k_i-1}) \in \partial f_1(\textbf{x}^{k_i+1})$ and $\sigma_1$-convexity of $f_1$ that
	\begin{equation}
	\label{eq:002}
	f_1(\bar{\textbf{x}}) \geq f_1(\textbf{x}^{k_i+1}) + \langle \textbf{y}^{k_i} + \gamma(\textbf{x}^{k_i}-\textbf{x}^{k_i-1}),\bar{\textbf{x}} - \textbf{x}^{k_i+1} \rangle + \frac{\sigma_1}{2}\Vert \bar{\textbf{x}} - \textbf{x}^{k_i+1} \Vert^2.
	\end{equation}
	Then we have
	\begin{align*}
	\zeta =&\lim\limits_{k\rightarrow \infty} f(\textbf{x}^k) \\
	=& \lim\limits_{i\rightarrow \infty} f_1(\textbf{x}^{k_i+1})-f_2(\textbf{x}^{k_i+1}) \\
	\overset{\eqref{eq:001}}{\leq}& \limsup\limits_{i\rightarrow \infty} f_1(\textbf{x}^{k_i+1})-\{f_2(\textbf{x}^{k_i})+\langle \textbf{y}^{k_i},\textbf{x}^{k_i+1}-\textbf{x}^{k_i} \rangle +\frac{\sigma_2}{2}\Vert \textbf{x}^{k_i+1}-\textbf{x}^{k_i}\Vert^2\}\\
	\overset{\eqref{eq:002}}{\leq}& \limsup\limits_{i\rightarrow \infty} 
	f_1(\bar{\textbf{x}}) -\langle \textbf{y}^{k_i} +\gamma(\textbf{x}^{k_i}-\textbf{x}^{k_i-1}),\bar{\textbf{x}}-\textbf{x}^{k_i+1}\rangle - \frac{\sigma_1}{2}\Vert \textbf{x}^{k_i+1}-\bar{\textbf{x}}\Vert^2\\ & -f_2(\textbf{x}^{k_i})-\langle \textbf{y}^{k_i},\textbf{x}^{k_i+1}-\textbf{x}^{k_i}\rangle-\frac{\sigma_2}{2}\Vert \textbf{x}^{k_i+1}-\textbf{x}^{k_i}\Vert^2 \\
	 =& f(\bar{\textbf{x}}),
	\end{align*}
	where the last equality follows that $\lim\limits_{i\rightarrow \infty} \Vert \textbf{x}^{k_i}-\textbf{x}^{k_i-1}\Vert = 0$, $f_2$ is continuous over $\dom f_1$, and $\{\textbf{y}^{k_i}\}$ is bounded. 
	On the other hand, because $f$ is a closed function, then we have
	\begin{align*}
	f(\bar{\textbf{x}}) &= f_1(\bar{\textbf{x}}) - f_2(\bar{\textbf{x}}) \\
	&\leq \liminf\limits_{i\rightarrow \infty} f_1(\textbf{x}^{k_i}) - f_2(\textbf{x}^{k_i})\\
	&= \liminf\limits_{i\rightarrow \infty} f(\textbf{x}^{k_i}) =\zeta.
	\end{align*}
	Thus, we obtain that $f(\bar{\textbf{x}}) = \zeta$ and conclude that $f\equiv \zeta$ on $\Omega_1$. \qed
\end{proof}

We can obtain a similar proposition for \RInDCAine\ described as follows whose proof is omitted.
\begin{proposition}
\label{prop:3.2}
	Suppose that the assumptions in Theorem \ref{them:2} hold. Let {\normalfont$\{\textbf{x}^k\}$} be the sequence generated by \RInDCAine\ and let $\Omega_2$ be the set of the limit points of {\normalfont$\{\textbf{x}^k\}$}. If {\normalfont$\{\textbf{x}^k\}$} is bounded, then we have
	\begin{itemize}
		\item[(i)] {\normalfont$\lim\limits_{k\rightarrow \infty} f(\textbf{x}^k):= \widetilde{\zeta}$} exists;
		\item[(ii)]  $f \equiv\widetilde{\zeta}$ on $\Omega_2$.
	\end{itemize}
\end{proposition}

\section{Sequential convergence analysis}
\label{sec:5}
In this section, we will establish the sequential convergence of RInDCA. Let us consider the next auxiliary functions:
\begin{align}
	 \label{eq:3.2.01}
	     &E({\normalfont\textbf{x},\textbf{y},\textbf{z}}) = f_1({\normalfont\textbf{x}})  -\langle {\normalfont\textbf{x},\textbf{y}}\rangle + f_2^*({\normalfont\textbf{y}}) + \frac{\sigma_1- \gamma }{2}\Vert {\normalfont\textbf{x}-\textbf{z}}\Vert^2,\\
	     &\widetilde{E}(\vx,\vy,\vz) = f_1(\vx) - \langle \vx,\vy\rangle + f_2^*(\vy) + \frac{\bar{\sigma}-\sigma_2-\gamma}{2}\Vert \vx -\vz\Vert^2,
\end{align}
where $f_1$ and $f_2$ are defined in problem \eqref{P} with the strong convexity parameters $\sigma_1$ and $\sigma_2$, $f_2^*$ is the conjugate function of $f_2$, and $\bar{\sigma}$ is defined in Lemma \ref{lem:3.2}. Note that our construction of $E$ and $\widetilde{E}$ is motivated by \cite{liu2019refined}, and 
our sequential convergence analysis follows similar arguments as in \cite{attouch2013convergence,liu2019refined}.  

First, the following Proposition \ref{prop:4.1} establishes the non-increasing property of the sequence $\{E({\normalfont\textbf{x}}^k,{\normalfont\textbf{y}}^{k-1},{\normalfont\textbf{x}}^{k-1})\}_{k\geq 1}$. 

\begin{proposition} 
	\label{prop:4.1}
	 Suppose that the assumptions in Theorem \ref{them:1} hold. Let $E$ be defined as in \eqref{eq:3.2.01}, ${\normalfont\{\textbf{x}^k\}}$ be the sequence generated by \RInDCAe, and ${\normalfont\textbf{y}^k} \in \partial f_2({\normalfont\textbf{x}^k})$ for all $k\geq 0$. Then we have
	\begin{enumerate}
		\item[(i)] For any $k\geq 1$, 
		\begin{equation}
		\label{eq:de2}
		E({\normalfont\textbf{x}}^{k+1},{\normalfont\textbf{y}}^k,{\normalfont\textbf{x}}^k) \leq E({\normalfont\textbf{x}}^k,{\normalfont\textbf{y}}^{k-1},{\normalfont\textbf{x}}^{k-1}) - \frac{\sigma_1+\sigma_2-2\gamma}{2}\Vert {\normalfont\textbf{x}}^k-{\normalfont\textbf{x}}^{k-1}\Vert^2;
		\end{equation}
		\item[(ii)] If {\normalfont$\{\textbf{x}^k\}$} is bounded, let $\Upsilon$ be the set of all limit points of the sequence ${\normalfont\{(\textbf{x}^k,\textbf{y}^{k-1},\textbf{x}^{k-1})\}}$. Then $\Upsilon$ is a nonempty  compact set, $$\lim\limits_{k\rightarrow\infty} E({\normalfont\textbf{x}^k,\textbf{y}^{k-1},\textbf{x}^{k-1})}:=\zeta$$ exists, and $E\equiv \zeta$ on $\Upsilon$.
	\end{enumerate}
\end{proposition}

\begin{proof}
	$(i)$ For any $k\geq 1$, $\textbf{y}^{k-1}\in \partial f_2(\textbf{x}^{k-1})$ together with $\sigma_2$-convexity of $f_2$ implies that 
	\begin{equation}
	\label{eq:200}
	f_2(\textbf{x}^k)\geq f_2(\textbf{x}^{k-1})+\langle \textbf{y}^{k-1},\textbf{x}^k-\textbf{x}^{k-1}\rangle + \frac{\sigma_2}{2}\Vert \textbf{x}^k-\textbf{x}^{k-1}\Vert^2.
	\end{equation}
	Thus we have
	\begin{align}
	\label{eq:3.5.10}
	\langle \textbf{y}^{k-1},\textbf{x}^k\rangle - f_2(\textbf{x}^k) + \frac{\sigma_2}{2}\Vert \textbf{x}^k-\textbf{x}^{k-1}\Vert^2 &\overset{\eqref{eq:200}}{\leq} \langle  \textbf{y}^{k-1},\textbf{x}^{k-1}\rangle - f_2(\textbf{x}^{k-1}) \nonumber\\
	&\ = f_2^*(\textbf{y}^{k-1}),
	\end{align}
	where the equality is implied by $\textbf{y}^{k-1}\in\partial f_2(\textbf{x}^{k-1})$, and thus $\langle \textbf{y}^{k-1},\textbf{x}^{k-1}\rangle = f_2(\textbf{x}^{k-1}) + f_2^*(\textbf{y}^{k-1})$ (see Theorem 4.20 in \cite{Beck_1order}). Moreover, $\textbf{y}^k +\gamma(\textbf{x}^k-\textbf{x}^{k-1}) \in\partial f_1(\textbf{x}^{k+1})$ and $\sigma_1$-convexity of $f_1$ imply that
    \begin{equation}
    \label{eq:3.5.11}
    f_1(\textbf{x}^k) \geq  f_1(\textbf{x}^{k+1}) + \langle \textbf{y}^k +\gamma(\textbf{x}^k-\textbf{x}^{k-1}),\textbf{x}^k-\textbf{x}^{k+1} \rangle + \frac{\sigma_1}{2}\Vert \textbf{x}^k-\textbf{x}^{k+1}\Vert^2.
    \end{equation}	
	Next, we have
	\begin{align*}
	&E(\textbf{x}^{k+1},\textbf{y}^k,\textbf{x}^k) \\
	=& f_1(\textbf{x}^{k+1})-\langle \textbf{x}^{k+1},\textbf{y}^k\rangle + f_2^*(\textbf{y}^k) + \frac{\sigma_1-\gamma}{2}\Vert \textbf{x}^{k+1}-\textbf{x}^k\Vert^2 \\
	\overset{\eqref{eq:3.5.11}}{\leq}& f_1(\textbf{x}^k) - \langle \textbf{y}^k + \gamma(\textbf{x}^k-\textbf{x}^{k-1}),\textbf{x}^k-\textbf{x}^{k+1}\rangle -\frac{\sigma_1}{2}\Vert \textbf{x}^{k+1}-\textbf{x}^k\Vert^2 \\
	& -\langle \textbf{x}^{k+1},\textbf{y}^k\rangle + f_2^*(\textbf{y}^k) + \frac{\sigma_1-\gamma}{2}\Vert \textbf{x}^{k+1}-\textbf{x}^k\Vert^2 \\
	=& f_1(\textbf{x}^k) - \langle \textbf{y}^k + \gamma(\textbf{x}^k-\textbf{x}^{k-1}),\textbf{x}^k-\textbf{x}^{k+1}\rangle -\frac{\sigma_1}{2}\Vert \textbf{x}^{k+1}-\textbf{x}^k\Vert^2 \\
	& -\langle \textbf{x}^{k+1},\textbf{y}^{k}\rangle +\langle \textbf{y}^k,\textbf{x}^k\rangle -f_2(\textbf{x}^k) + \frac{\sigma_1-\gamma}{2}\Vert \textbf{x}^{k+1}-\textbf{x}^k\Vert^2 \\
	=& f_1(\textbf{x}^k) - f_2(\textbf{x}^k) - \langle \gamma(\textbf{x}^k-\textbf{x}^{k-1}),\textbf{x}^k-\textbf{x}^{k+1}\rangle-\frac{\gamma }{2}\Vert \textbf{x}^{k+1}-\textbf{x}^k\Vert^2 \\
	\leq & f_1(\textbf{x}^k) - f_2(\textbf{x}^k) +  \frac{\gamma}{2}\Vert \textbf{x}^k-\textbf{x}^{k-1}\Vert^2 \\
	\overset{\eqref{eq:3.5.10}}{\leq}& f_1(\textbf{x}^k) -\langle \textbf{x}^k,\textbf{y}^{k-1}\rangle + f_2^*(\textbf{y}^{k-1})  + \frac{\gamma-\sigma_2}{2}\Vert \textbf{x}^k-\textbf{x}^{k-1}\Vert^2  \\
	=& E(\textbf{x}^k,\textbf{y}^{k-1},\textbf{x}^{k-1}) - \frac{\sigma_1+\sigma_2-2\gamma}{2}\Vert \textbf{x}^{k}-\textbf{x}^{k-1}\Vert^2,
	\end{align*}
	where the second equality is implied by $\textbf{y}^k \in \partial f_2(\textbf{x}^k)$, and thus $f_2(\textbf{x}^k) + f_2^*(\textbf{y}^k) = \langle \textbf{x}^k,\textbf{y}^k\rangle$; the second inequality follows from $\langle \textbf{x}^k-\textbf{x}^{k-1},\textbf{x}^k-\textbf{x}^{k+1}\rangle \geq -(\Vert \textbf{x}^k-\textbf{x}^{k-1}\Vert^2 + \Vert \textbf{x}^k-\textbf{x}^{k+1}\Vert^2)/2$.\\
	$(ii)$ The boundedness of $\{\textbf{x}^k\}$ implies the boundedness of $\{\textbf{y}^{k-1}\}$. Thus, the sequence $\{(\textbf{x}^k,\textbf{y}^{k-1},\textbf{x}^{k-1})\}$ is bounded, yielding that $\Upsilon$ is a nonempty  compact set. Let ${\normalfont(\bar{\textbf{x}},\bar{\textbf{y}},\bar{\textbf{z}})} \in \Upsilon$, it is easy to see that $(\bar{\textbf{x}},\bar{\textbf{y}},\bar{\textbf{z}}) \in \Upsilon$ implies $\bar{\textbf{x}} = \bar{\textbf{z}} \in \Omega_1$.\\
	For any $k\geq 1$, we have
	\begin{align*}
	    E(\textbf{x}^k,\textbf{y}^{k-1},\textbf{x}^{k-1}) &= f_1({\normalfont\textbf{x}^k})  -\langle {\normalfont\textbf{x}^k,\textbf{y}^{k-1}}\rangle + f_2^*({\normalfont\textbf{y}^{k-1}}) + \frac{\sigma_1- \gamma }{2}\Vert {\normalfont\textbf{x}^k-\textbf{x}^{k-1}}\Vert^2 \\
	    &\geq f_1(\textbf{x}^k) - f_2(\textbf{x}^k) + \frac{\sigma_1- \gamma }{2}\Vert {\normalfont\textbf{x}^k-\textbf{x}^{k-1}}\Vert^2 \\
	    &\geq f^*  + \frac{\sigma_1- \gamma }{2}\Vert {\normalfont\textbf{x}^k-\textbf{x}^{k-1}}\Vert^2,
	\end{align*}
	where the first inequality follows from $f_2(\textbf{x}^k) + f_2^*(\textbf{y}^{k-1}) \geq \langle \textbf{x}^k,\textbf{y}^{k-1}\rangle$ and $f^*$ is the lower bound of $f$. Thus, $\liminf\limits_{k\rightarrow\infty} E(\textbf{x}^k,\textbf{y}^{k-1},\textbf{x}^{k-1}) \geq f^*$. Then, combining the fact that the sequence $\{E({\normalfont\textbf{x}}^k,{\normalfont\textbf{y}}^{k-1},{\normalfont\textbf{x}}^{k-1})\}$ is nonincreasing (statement (i)), we obtain the existence of
	$\lim\limits_{k\rightarrow\infty} E({\normalfont\textbf{x}^k,\textbf{y}^{k-1},\textbf{x}^{k-1})}$. Next, we prove the last part of (ii). Given any $(\bar{\textbf{x}},\bar{\textbf{y}},\bar{\textbf{x}}) \in \Upsilon$, there exists a subsequence $(\textbf{x}^{k_i},\textbf{y}^{k_i-1},\textbf{x}^{k_i-1})$ such that $$\lim\limits_{i\rightarrow\infty} \Vert (\textbf{x}^{k_i},\textbf{y}^{k_i-1},\textbf{x}^{k_i-1})-(\bar{\textbf{x}},\bar{\textbf{y}},\bar{\textbf{x}})\Vert = 0.$$
	Then we have
	\begin{align*}
	&\lim\limits_{k\rightarrow\infty}  E(\textbf{x}^k,\textbf{y}^{k-1},\textbf{x}^{k-1}) \\ =& \lim\limits_{i\rightarrow\infty} E(\textbf{x}^{k_i+1},\textbf{y}^{k_i},\textbf{x}^{k_i}) \\
	=& \lim\limits_{i\rightarrow \infty} f_1(\textbf{x}^{k_i+1}) - \langle \textbf{x}^{k_i+1},\textbf{y}^{k_i}\rangle+f_2^*(\textbf{y}^{k_i}) + \frac{\sigma_1-\gamma}{2}\Vert \textbf{x}^{k_i+1}-\textbf{x}^{k_i}\Vert^2 \\
	=& \lim\limits_{i\rightarrow \infty}
	f_1(\textbf{x}^{k_i+1}) - \langle \textbf{x}^{k_i+1},\textbf{y}^{k_i}\rangle - f_2(\textbf{x}^{k_i}) + \langle \textbf{y}^{k_i},\textbf{x}^{k_i}\rangle + \frac{\sigma_1-\gamma}{2}\Vert \textbf{x}^{k_i+1}-\textbf{x}^{k_i}\Vert^2 \\
	\leq& \limsup\limits_{i\rightarrow \infty} f_1(\textbf{x}^{k_i}) -\langle \textbf{y}^{k_i} +\gamma(\textbf{x}^{k_i}-\textbf{x}^{k_i-1}),\textbf{x}^{k_i}-\textbf{x}^{k_i+1}\rangle - \frac{\sigma_1}{2}\Vert \textbf{x}^{k_i+1}-\textbf{x}^{k_i}\Vert^2  \\
	& - \langle \textbf{x}^{k_i+1},\textbf{y}^{k_i}\rangle - f_2(\textbf{x}^{k_i}) + \langle \textbf{y}^{k_i},\textbf{x}^{k_i}\rangle + \frac{\sigma_1-\gamma}{2}\Vert \textbf{x}^{k_i+1}-\textbf{x}^{k_i}\Vert^2 \\
	=& \limsup\limits_{i\rightarrow\infty} f(\textbf{x}^{k_i}) -\gamma\langle \textbf{x}^{k_i}-\textbf{x}^{k_i-1},\textbf{x}^{k_i}-\textbf{x}^{k_i+1}\rangle - \frac{\gamma}{2}\Vert \textbf{x}^{k_i+1}-\textbf{x}^{k_i}\Vert^2\\
	=& \limsup\limits_{i\rightarrow\infty} f(\textbf{x}^{k_i}) \\
	=& \zeta \\
	=& f(\bar{\textbf{x}}) \leq  E(\bar{\textbf{x}},\bar{\textbf{y}},\bar{\textbf{x}}),
	\end{align*} 
	where the third equality follows from $\textbf{y}^{k_i} \in \partial f_2(\textbf{x}^{k_i})$ and thus $f_2(\textbf{x}^{k_i}) + f_2^*(\textbf{y}^{k_i}) = \langle \textbf{y}^{k_i},\textbf{x}^{k_i}\rangle$; the first inequality is implied by $\textbf{y}^{k_i} + \gamma (\textbf{x}^{k_i}-\textbf{x}^{k_i-1}) \in \partial f_1(\textbf{x}^{k_i+1})$ and $\sigma_1$-convexity of $f_1$; the fifth equality follows from $\lim\limits_{i\rightarrow\infty} \Vert \textbf{x}^{k_i}-\textbf{x}^{k_i-1}\Vert=0$. 
	On the other hand, the closedness of $E$ implies that $E(\bar{\textbf{x}},\bar{\textbf{y}},\bar{\textbf{x}}) \leq \zeta.$
	Thus, we obtain that $E(\bar{\textbf{x}},\bar{\textbf{y}},\bar{\textbf{x}}) = \zeta$ and conclude that $E\equiv \zeta$ on $\Upsilon$.\qed
\end{proof}

Next, we start to create the sequential convergence of \RInDCAe\ based on the KL property of $E$ and Proposition \ref{prop:4.1}. 
Note that the KL property of $E$ can be implied by the KL property of $f$ (see \cite{liu2019refined}).

\begin{theorem}
	Suppose that the assumptions in Proposition \ref{prop:4.1} hold. Let $E$ defined in \eqref{eq:3.2.01} be a KL function, ${\normalfont\{\textbf{x}^k\}}$ be the sequence generated by \RInDCAe, and ${\normalfont\textbf{y}^k} \in \partial f_2({\normalfont\textbf{x}^k})$ for all $k\geq 0$. Then we have
	\begin{itemize}
	    \item[(i)] There exists $D >0$ such that $\forall k\geq 1$,  
		\begin{equation}
		\label{eq:them001}
		{\normalfont\text{dist}}({\normalfont\zero},\partial E({\normalfont\textbf{x}^k,\textbf{y}^{k-1},\textbf{x}^{k-1})}) \leq D(\Vert {\normalfont\textbf{x}^k-\textbf{x}^{k-1}}\Vert+\Vert {\normalfont\textbf{x}^{k-1}-\textbf{x}^{k-2}}\Vert).
		\end{equation}
	    \item[(ii)] If ${\normalfont\{\textbf{x}^k\}}$ is bounded, then ${\normalfont\{\textbf{x}^k\}}$ converges to a critical point of problem (\ref{P}).
	\end{itemize}
\end{theorem}

\begin{proof}
    $(i)$  Note that the subdifferential of $E$ at $\vw^k :=(\textbf{x}^k,\textbf{y}^{k-1},\textbf{x}^{k-1})$ is:
	\begin{equation*}
	\partial E(\vw^k) = 
	\begin{bmatrix}
	\partial f_1(\textbf{x}^k)-\textbf{y}^{k-1} + \frac{\gamma-\sigma_2}{2}(\textbf{x}^k-\textbf{x}^{k-1})	 \\ 
	-\textbf{x}^k + \partial f_2^*(\textbf{y}^{k-1})	\\
	-\frac{\gamma-\sigma_2}{2}(\textbf{x}^k-\textbf{x}^{k-1})   \\
	\end{bmatrix}.
	\end{equation*}
	Because $\textbf{y}^{k-1}\in \partial f_2(\textbf{x}^{k-1})$, thus $\textbf{x}^{k-1} \in \partial f_2^*(\textbf{y}^{k-1})$; besides, we have $\textbf{y}^{k-1} + \gamma (\textbf{x}^{k-1}-\textbf{x}^{k-2}) \in \partial f_1(\textbf{x}^k)$. Combining these relations, we have 
	\begin{equation*}
	\begin{bmatrix}
	\gamma(\textbf{x}^{k-1}-\textbf{x}^{k-2}) +  \frac{\gamma-\sigma_2}{2}(\textbf{x}^k-\textbf{x}^{k-1})	 \\ 
	-\textbf{x}^k + \textbf{x}^{k-1}	\\
	-\frac{\gamma-\sigma_2}{2}(\textbf{x}^k-\textbf{x}^{k-1})   \\
	\end{bmatrix} \in \partial E(\vw^k). 
	\end{equation*}
	Thus, it is easy to see that there exists $D >0$ such that $\forall k\geq 1$,
	 \begin{equation*}
	 \text{dist}(\zero,\partial E(\textbf{x}^k,\textbf{y}^{k-1},\textbf{x}^{k-1})) \leq D(\Vert \textbf{x}^k-\textbf{x}^{k-1}\Vert+\Vert \textbf{x}^{k-1}-\textbf{x}^{k-2}\Vert).
	 \end{equation*}\\
	 $(ii)$ It is sufficient to prove that $\sum_{k=1}^{\infty} \Vert \textbf{x}^k-\textbf{x}^{k-1}\Vert < \infty$. If there exists $N_0 \geq 1$ such that $E(\vw^{N_0})=\zeta_1$, then the inequality \eqref{eq:de2} implies that $\vx^k = \vx^{N_0}, \forall k\geq N_0$; naturally, the sequence $\{\vx^k\}$ is convergent. Thus, we can assume that $$E(\vw^k) > \zeta, \forall k\geq 1.$$ Then, it follows from $\Upsilon$ is a compact set, $ \Upsilon \subseteq \dom \partial E$ and $E\equiv \zeta$ on $\Upsilon$ that there exists  $\varepsilon >0$, $\eta >0$ and $\varphi\in\Xi_{\eta}$ such that 
	\begin{equation*}
	\varphi'(E(\textbf{x},\textbf{y},\textbf{z})-\zeta) \text{dist}(\zero,\partial E(\textbf{x},\textbf{y},\textbf{z})) \geq 1
	\end{equation*} 
	for all $(\textbf{x},\textbf{y},\textbf{z}) \in U$ with
	\begin{equation*}
	U = \{(\textbf{x},\textbf{y},\textbf{z}):\text{dist}((\textbf{x},\textbf{y},\textbf{z}),\Upsilon)<\varepsilon \} \cap \{(\textbf{x},\textbf{y},\textbf{z}):\zeta <E(\textbf{x},\textbf{y},\textbf{z})<\zeta +\eta\}.
	\end{equation*}
	Moreover, it is easy to derive from $\lim\limits_{k\rightarrow\infty}\text{dist}(\vw^k,\Upsilon) = 0$ and $\lim\limits_{k\rightarrow\infty}E(\vw^k) = \zeta$ that there exists $N\geq 1$ such that $\vw^k \in U, \forall k\geq N$. Thus, 
	\begin{equation*}
	\varphi'(E(\vw^k)-\zeta) \text{dist}(\textbf{0},\partial E(\vw^k) \geq 1,~ \forall k\geq N.
	\end{equation*} 
	Using the concavity of $\varphi$, we see that $\forall k\geq N$,
	\begin{align}\label{eq:thm:4.2_01}
	& [\varphi(E(\vw^k)-\zeta)-\varphi(E(\vw^{k+1})-\zeta)]
	\text{dist}(\zero,\partial E(\vw^k))\nonumber\\
	\geq& \varphi'(E(\vw^k)-\zeta) \text{dist}(\zero,\partial E(\vw^k)) (E(\vw^k)-E(\vw^{k+1}))\nonumber\\
	\geq &   E(\vw^k)-E(\vw^{k+1}).
	\end{align}
	Let $\Delta_k := \varphi(E(\vw^k)-\zeta)-\varphi(E(\vw^{k+1})-\zeta)$ and $C:= \frac{\sigma_1+\sigma_2-2\gamma}{2}\Vert \textbf{x}^k-\textbf{x}^{k-1}\Vert^2$. Then combining \eqref{eq:thm:4.2_01}, \eqref{eq:them001} and (i) of Proposition \ref{prop:3.2}, we have 
	\begin{equation*}
	\Vert \textbf{x}^k-\textbf{x}^{k-1}\Vert^2 \leq \frac{D}{C}\Delta_k(\Vert \textbf{x}^k-\textbf{x}^{k-1}\Vert+\Vert \textbf{x}^{k-1}+\textbf{x}^{k-2}\Vert).
	\end{equation*}
	Thus, taking square roots on both sides, we obtain
	\begin{align*}
	\Vert \textbf{x}^k-\textbf{x}^{k-1}\Vert &\leq \sqrt{\frac{2D}{C}\Delta_k} \sqrt{\frac{\Vert \textbf{x}^k-\textbf{x}^{k-1}\Vert+\Vert \textbf{x}^{k-1}-\textbf{x}^{k-2}\Vert}{2}} \\
	&\leq \frac{D}{C}\Delta_k + \frac{1}{4}\Vert \textbf{x}^k-\textbf{x}^{k-1}\Vert + \frac{1}{4}\Vert \textbf{x}^{k-1}-\textbf{x}^{k-2}\Vert,
	\end{align*}
	this yields
	\begin{equation*}
	\frac{1}{2}\Vert \textbf{x}^k-\textbf{x}^{k-1}\Vert \leq \frac{D}{C}\Delta_k + \frac{1}{4}\Vert \textbf{x}^{k-1}-\textbf{x}^{k-2}\Vert - \frac{1}{4}\Vert \textbf{x}^k - \textbf{x}^{k-1}\Vert.
	\end{equation*}
	Considering also that $\sum_{k=1}^{\infty} \Delta_k < \infty$, thus the above inequality implies that $\sum_{k=1}^{\infty} \Vert \textbf{x}^k-\textbf{x}^{k-1}\Vert < \infty$, which indicates that $\{\vx^k\}$ is a Cauchy sequence, and thus convergent. The proof is completed. \qed
\end{proof}

The sequential convergence of \RInDCAine\ can be similarly developed based on the KL property of $\widetilde{E}$, Proposition \ref{prop:3.2}, and Proposition \ref{prop:4.1} for $\widetilde{E}$, thus the proof is omitted.

\section{Numerical results}\label{sec:6}
In this part, we conduct numerical simulations on testing copositivity of matrices (constrained case) and image denoising problem (unconstrained case). All experiments are implemented in Matlab 2019a on a 64-bit PC with an Intel(R) Core(TM) i5-6200U CPU (2.30GHz) and 8GB of RAM. 

\subsection{Checking copositivity of matrices}
Recall that a symmetric matrix $\textbf{A} \in \R^{n\times n}$ is called copositive if $\textbf{x}^{\top}\textbf{A}\textbf{x} \geq 0$ for all $\textbf{x}\geq 0$, and called non-copositive otherwise. Consider the next formulation
\begin{equation}
\label{copo}
\tag{COPO}
\begin{split}
\min\quad & \frac{1}{2}\textbf{x}^{\top}\textbf{A}\textbf{x} \\
\text{s.t.}\quad  & \textbf{x}\in \R_{+}^n.\\
\end{split}
\end{equation}
Clearly, $\textbf{A}$ is copositive if and only if the optimal value of problem (\ref{copo}) is nonnegative, and problem (\ref{copo}) is equivalent to the next DC program
\begin{equation}
\label{DC_COPO}
\tag{DC$_{copo}$}
    f(\textbf{x}) = \underbrace{\frac{L}{2}\Vert \textbf{x}\Vert^2 + \delta_{\R_+^n}(\textbf{x})}_{f_1(\textbf{x})} - (\underbrace{\frac{L}{2}\Vert \textbf{x}\Vert^2 - \frac{1}{2}\textbf{x}^{\top}\textbf{A}\textbf{x}}_{f_2(\textbf{x})}),
\end{equation}
where $L\geq \Vert \textbf{A}\Vert$.
In \cite{JP_copo}, DCA applied for the above DC decomposition 
is used to check the copositivity of $\textbf{A}$. For our experiments, the same instances of $\textbf{A}$ as in \cite{JP_copo} will be used. A brief introduction about the  instances is given as follows. Let $\textbf{E}_n\in \R^{n\times n}$ be the matrix with all entries being one. The matrix $\textbf{A}_{\text{cycle}} = (a_{ij})\in \R^{n\times n}$ is given component-wise by
\begin{equation*}
a_{ij} :=
\begin{cases}
1 & \text{if $\vert i-j\vert \in \{1,n-1\}$},\\
0 & \text{otherwise}.
\end{cases}
\end{equation*}
It was known that the following matrix
\begin{equation*}
\textbf{Q}_n^{\mu} := \mu (\textbf{E}_n-\textbf{A}_{\text{cycle}}) - \textbf{E}_n \in \R^{n\times n}
\end{equation*}
is copositive for $\mu \geq 2$ and non-copositive for $\mu < 2$. When $\mu = 2$, $\textbf{Q}_n^{\mu}$ is called Horn matrix, denoted as $\textbf{H}_n$.

\redtext{In order to apply DCA, \InDCAe, and \RInDCAe\ for \eqref{DC_COPO}, we set $L = \Vert\textbf{A}\Vert$ for DCA and $L = \Vert\textbf{A}\Vert+1$ for \InDCAe\ and \RInDCAe.} The inertial ranges of \InDCAe\ and \RInDCAe\ are $[0,1/2)$ and $[0,(\Vert \textbf{A}\Vert+2)/2)$, respectively. For \InDCAe, $\gamma = 0.499$ is set, while for \RInDCAe\, we choose $\gamma = 0.499(\Vert \textbf{A}\Vert+2)$. 
Moreover, the common initial point for all involved algorithms is randomly generated as follows: we first use the standard normal distribution to generate a point $\textbf{y} \in \R^n$, then we obtain the initial point $\textbf{x}^0 = (\textbf{x}^0_j)$ by setting $$\textbf{x}^0_j = \frac{e^{\textbf{y}_j}}{\sum_{j=1}^{n} e^{\textbf{y}_j}}$$ for $j = 1,\cdots,n$.  We report the number of iterations, the computational time, and the  objective functions values. The bold values in Tables \ref{tab:b} and \ref{tab:c} highlight the best numerical results. The computational time presented is the average execution time (in seconds) over 10 runs. 

\begin{table}
	\centering
	\caption{Performances of DCA, \InDCAe, and \RInDCAe\ for checking the copositivity of $\textbf{H}_n$ with $n \in \{500,1000,1500,2000\}$}
	\label{tab:b}
	\begin{tabular}{c|c|c|c|c} 
		\hline
		Size (n) &
		Algorithm & It. & Time & Fval   \\
		\hline\hline
		\multirow{3}{.5in}{500}& DCA & 1963 & 0.0799 & 4.0483e-14  \\
		& \InDCAe & 1965 & 0.0817& 4.0470e-14 \\
		& \RInDCAe & \textbf{1020} & \textbf{0.0399} & \textbf{9.6531e-15} \\
		\hline
		\multirow{3}{.5in}{1000}& DCA & 2915 & 1.1634 & 1.6403e-13  \\
		& \InDCAe & 2916 & 1.1692 & 1.6449e-13 \\
		& \RInDCAe & \textbf{1562} & \textbf{0.6276} & \textbf{3.9847e-14} \\
		\hline
		\multirow{3}{.5in}{1500}& DCA & 4772 &  4.9034   & 3.7082e-13  \\
		& \InDCAe & 4773 & 4.9557 & 3.7170e-13 \\
		& \RInDCAe & \textbf{2542} & \textbf{2.5938} & \textbf{9.1359e-14} \\
		\hline
		\multirow{3}{.5in}{2000}& DCA & 5829 & 10.8607  & 6.6003e-13  \\
		& \InDCAe & 5830 & 11.1934 & 6.6094e-13 \\
		& \RInDCAe & \textbf{3129} & \textbf{6.0101} & \textbf{1.6295e-13} \\
		\hline
	\end{tabular}
\end{table}

First, we consider to check the copositivity of the Horn matrices.  All algorithms are terminated when $\Vert \textbf{x}^k-\textbf{x}^{k-1}\Vert < 10^{-9}$. Table \ref{tab:b} describes the performances of DCA, \InDCAe, and \RInDCAe\ for checking the copositivity of $\textbf{H}_n$ with $n \in \{500,1000,1500,2000\}$. It is observed that DCA and \InDCAe\ perform almost the same since the number of iterations of \InDCAe\ is very close to that of DCA. We also observe that \RInDCAe\ performs the best among the three algorithms, terminating with almost half of the number of iterations of DCA or \InDCAe, and with the smallest objective function values. The reason why \InDCAe\ loses the inertial-force effect may be that 0.499 is much smaller than $0.499(\Vert \textbf{A}\Vert+2)$ for the instance $\textbf{H}_n$ when $n$ is large.

Next, we check the non-copositivity of $\textbf{Q}_n^{\mu}$ $(\mu = 1.9)$. Here, the stopping criterion for the involved algorithms DCA, \InDCAe, and \RInDCAe\ is that  $f(\textbf{x}^{k}) \leq -10^{-6}$. Once this condition is verified, then the non-copositivity of the involved matrix is determined. Table \ref{tab:c} shows the performances of DCA, \InDCAe, and \RInDCAe\ for checking the copositivity of $\textbf{Q}_n^{\mu}$ ($\mu = 1.9$) with $n \in \{500,1000,1500,2000\}$.  
\begin{table}[h]
	\begin{center}
	\caption{Performances of DCA, \InDCAe, and \RInDCAe\ for checking the copositivity of $\textbf{Q}_n^{\mu}$ ($\mu = 1.9$) with $n \in \{500,1000,1500,2000\}$}
	
		\begin{tabular}{c|c|c|c|c} 
			\hline
			Size (n) &
			Algorithm & It. & Time & Fval   \\
			\hline\hline
			\multirow{3}{.5in}{500}& DCA & 430 & 0.0330 & -9.9882e-07  \\
			& \InDCAe & 430 & 0.0359 & -9.9139e-07 \\
			& \RInDCAe & \textbf{209} & \textbf{0.0169} & -9.8189e-07 \\
			\hline
			\multirow{3}{.5in}{1000}& DCA & 824 & 0.5786 & -9.9785e-07   \\
			& \InDCAe & 825 & 0.6014 & -9.9907e-07 \\
			& \RInDCAe & \textbf{405} & \textbf{0.2991} & -9.9934e-07 \\
			\hline
			\multirow{3}{.5in}{1500}& DCA & 2036 & 4.6078 & -9.9965e-07  \\
			& \InDCAe & 2037 & 4.6695 & -9.9975e-07 \\
			& \RInDCAe & \textbf{1021} & \textbf{2.3593} & -9.9985e-07 \\
			\hline
			\multirow{3}{.5in}{2000}& DCA & 5094 & 14.8749 & -9.9993e-07  \\
			& \InDCAe & 5095 & 15.1039 & -9.9987e-07 \\
			& \RInDCAe & \textbf{2559} & \textbf{7.7138} & -9.9977e-07 \\
			\hline
\end{tabular}
\label{tab:c}
	\end{center}
\end{table}
It is also observed that among the involved algorithms \RInDCAe\ performs the best, hitting the stopping criterion $f(\textbf{x}^{k}) \leq -10^{-6}$ with about half of the number of iterations as well as the time of DCA or \InDCAe. The behavior of the algorithms for checking the noncopositivity almost coincides with the previous experiment.

The above experiments show that the effect of heavy-ball inertial-force may be lost if the strong convexity of the first DC component is not considered.

\subsection{Image denoising problem}
Consider a gray-scale image of $m\times n$ pixels and with entries in $[0,1]$, where 0 represents pure black and 1 represents pure white.  The original image $\textbf{X}\in\R^{m\times n}$ is assumed to be corrupted with a noise $\textbf{B}$. Then, we obtain the observed image $\widetilde{\textbf{X}} = \textbf{X} +\textbf{B}$, which is known to us. We aim to recover the original image $\textbf{X}$ by solving the next problem
\begin{align}
\label{im_prob}
    \min \{\frac{\mu}{2}\Vert \bX- \widetilde{\textbf{X}} \Vert^2_{F} + \text{TV}_{\phi}(\bX):\bX \in\R^{m\times n}  \},
\end{align}
where $\mu>0$ is a trade-off constant; $\widetilde{\textbf{X}}$ is the observed image; $\Vert \cdot \Vert_F$ represents the Frobenius norm; $\phi:\R_{+} \rightarrow \R$ is a concave function and $\text{TV}_{\phi}$ is defined as:
\begin{equation*}
    \text{TV}_{\phi}(\textbf{X}) = \sum_{i,j} \phi(\Vert \nabla \bX_{i,j}\Vert)
\end{equation*}
with
\begin{equation*}
\Vert \nabla \bX_{i,j}\Vert=
\begin{cases}
	\vert \textbf{X}_{i,j}-\textbf{X}_{i,j+1}\vert& i=m,j\neq n, \\
	\vert \textbf{X}_{i,j}-\textbf{X}_{i+1,j}\vert& i\neq m,j=n, \\
	0                                  & i=m,j=n,                \\
	\sqrt{(\textbf{X}_{i,j}-\textbf{X}_{i+1,j})^2+(\textbf{X}_{i,j}-\textbf{X}_{i,j+1})^2}& \text{otherwise}.\\
\end{cases}
\end{equation*}
When $\phi \equiv 1$, we denote $\text{TV}$ as $\text{TV}_{\phi}$. The model \eqref{im_prob} with $\phi \equiv 1$ is convex. In \cite{de2019inertial}, it has shown that the convex model ($\phi \equiv 1$) is not effective to preserve edges in the restoration process of piecewise-constant images than the model \eqref{im_prob} with $\phi$ being concave, such as the $\phi$ described in Table \ref{tab:1}. Moreover it has also shown how to derive the DC decomposition for $\text{TV}_{\phi}$. Concretely, $\text{TV}_{\phi}(\textbf{X})$ has the following DC decomposition $$\text{TV}_{\phi}(\textbf{X}) = \text{TV}(\textbf{X}) - (\text{TV}(\textbf{X}) - \text{TV}_{\phi}(\textbf{X})).$$ Then, we get from \eqref{im_prob} the next two equivalent DC programs 
    \begin{equation}
        \label{im_denoise}
        \min \{\underbrace{\frac{\mu}{2}\Vert \bX- \widetilde{\textbf{X}} \Vert^2_{F} + \text{TV}(\textbf{X})}_{f_1(\textbf{X})} - (\underbrace{\text{TV}(\textbf{X}) - \text{TV}_{\phi}(\bX)}_{f_2(\textbf{X})}):\bX \in\R^{m\times n}  \},
    \end{equation}
and 
\begin{equation}
	\label{im_denoise1}
	\min \{\underbrace{\frac{\mu+\rho}{2}\Vert \bX- \widetilde{\textbf{X}} \Vert^2_{F} + \text{TV}(\textbf{X})}_{\widetilde{f_1}(\textbf{X})} - (\underbrace{\frac{\rho}{2}\Vert \bX- \widetilde{\textbf{X}} \Vert^2_{F} + \text{TV}(\textbf{X}) - \text{TV}_{\phi}(\bX)}_{\widetilde{f}_2(\textbf{X})})\},
\end{equation}
where $\rho >0$.

\begin{table}[ht]
	\caption{Some examples of concave function $\phi_{a}:\R_{+}\rightarrow \R$ parameterized by $a>0$}
	\label{tab:1}      
	\centering
	\begin{tabular}{lllll}
		\hline\noalign{\smallskip}
		& $\phi_{\log}$ & $\phi_{\text{rat}}$ & $\phi_{\text{atan}}$ & $\phi_{\exp}$ \\
		\noalign{\smallskip}\hline\noalign{\smallskip}
		$\phi_a(r)$ & $\frac{\log(1+ar)}{a}$ & $\frac{r}{1+ar/2}$& $\frac{\text{atan}((1+ar)/\sqrt{3})-\pi/6}{a\sqrt{3}/4}$ & $\frac{1-\exp(-ar)}{a}$ \\
		$\phi_{a}^{'}(r)$ & $\frac{1}{1+ar}$ & $\frac{1}{(1+ar/2)^2}$ & $\frac{4}{a^2r^2+2ar+4}$ & $\frac{1}{\exp(ar)}$\\
		\noalign{\smallskip}\hline
	\end{tabular}
\end{table}

We will compare DCA, \InDCAe, \RInDCAe, and ADCA \cite{nhat2018accelerated} (a DCA based algorithm using the Nesterov's extrapolation). Note that DCA and ADCA are applied for \eqref{im_denoise}, while \InDCAe\ and \RInDCAe\ are applied for \eqref{im_denoise1}. 
All of the involved algorithms require solving the convex subproblems in form of 
\begin{equation*}
    \min\{ \frac{\widetilde{t}}{2}\Vert \textbf{X}-\textbf{Z}\Vert_F^2 + \text{TV}(\textbf{X}):\textbf{X}\in\R^{m\times n}\},
\end{equation*}
where $\widetilde{t}>0$ and $\textbf{Z}\in\R^{m\times n}$, which can be efficiently solved by the fast gradient projection (FGP) method proposed in \cite{beck2009fast}.

We use the gray-scale image checkerboard.png ($200\times 200$ pixels) for our experiments. White noise with variance 0.1 is added to the image to produce the observed image $\widetilde{\textbf{X}}$ in both \eqref{im_denoise} and \eqref{im_denoise1}. Here, we use the same method as described in \cite{de2019inertial} to generate the corrupted image. The original image and the corrupted image are shown in Fig. \ref{fig:denoise0}. To measure the quality of reconstruction images, we use the peak signal-to-noise ratio (PSNR) and the structural similarity (SSIM) index (see \cite{monga2017handbook}). In general, the larger PSNR means the better restoration, while the closer to 1 of SSIM indicates the better reconstruction quality. For all the involved algorithms, the observed images $\widetilde{\textbf{X}}$ are set as the common initial point, whose convex subproblems are all solved by FGP with the stopping condition that the distance between two consecutive iterates is smaller than $10^{-4}$. It is also worth noting that for problems \eqref{im_denoise} and \eqref{im_denoise1}, another DCA based method with Nesterov's extrapolation developed in \cite{wen2018proximal} reduces to DCA.

\begin{figure}[htbp]
  \subfigure[Original image]{\includegraphics[width=0.44\textwidth]{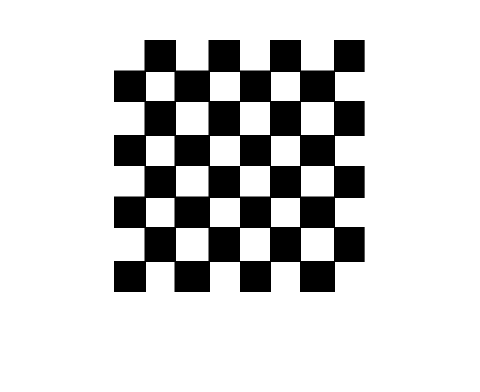}}
  \subfigure[Corrupted image]{\includegraphics[width=0.44\textwidth]{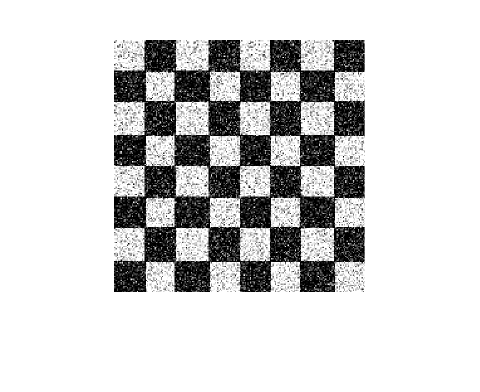}}
  \caption{Original and corrupted images}
\label{fig:denoise0}       
\end{figure}

\redtext{We test DCA and ADCA for \eqref{im_denoise}, \InDCAe\ and \RInDCAe\ for \eqref{im_denoise1}.} The parameter $\mu$ for both \eqref{im_denoise} and  \eqref{im_denoise1} ranges from 0.8 to 1.6, and we set $\phi = \phi_{\text{rat}}$ with $a=6$, moreover, the parameter $\rho$ in \eqref{im_denoise1} is fixed as 1. The inertial step-size $\gamma$ in \InDCAe\ is set as $0.5 \times \rho \times 99\% = 0.495$, and in \RInDCAe\ as $0.5\times(\mu+2\rho)\times 99\% = 0.495\mu+0.99$. Besides, the parameter $q$ in ADCA is set as 3.  In Table \ref{tab:denoise2}, we present in detail the best objective values $f_{\text{best}}$ within 100 iterations and the corresponding values of PSNR and SSIM. The trends of the objective function values together with the values of PSNR and SSIM when $\mu = 1.2$ with respect to running time (average on 10 runs) are plotted in Fig \ref{fig:denoise21}. 
Moreover, the recovered images within 100 iterations are demonstrated in Fig \ref{fig:denoise22}. Here, we observes that \RInDCAe\ always obtains the lowest objective function values, and almost always the best PSNR and SSIM. The benefit of enlarged inertial step-size is indicated by comparing \InDCAe\ and \RInDCAe. Moreover, we also observe that ADCA performs better than DCA and \InDCAe. Finally, by comparing DCA and \InDCAe, we observe that adding a strong convexity term to the original DC program \eqref{im_denoise} degrades the effectiveness of DCA-based algorithm, which can not be offset by the benefit of inertial-force. 

As a conclusion, for the image denoising problem, our method \RInDCAe\ is a promising approach, which outperforms the other tested algorithms.

\begin{table}[h!]
	\caption{Restoration of the corrupted checkerboard image}
	\label{tab:denoise2}
	\centering
	\begin{tabular}{c|c|ccc} \hline
		Algorithm & $\mu$ & $f_{\text{best}}$ & PSNR & SSIM \\
		\hline\hline
		DCA &  &   2.2920e+03	 & 24.2319 & 0.8268  \\
		\InDCAe &  0.8 &   2.2937e+03  &   24.2902 & 0.8269 \\
		\RInDCAe &  &   \textbf{2.2896e+03}	 & 24.4612 & \textbf{0.8324} \\
		ADCA &  &   2.2909e+03	 & \textbf{24.4654} & 0.8272\\
		\hline
		DCA &  &   2.4923e+03	 & 	24.1899 & 0.8324 \\
		\InDCAe &  0.9 &    2.4937e+03  &  24.1308 & 0.8307 \\
		\RInDCAe &  &   \textbf{2.4876e+03}	 & \textbf{24.3191} & \textbf{0.8363} \\
		ADCA &  &   2.4907e+03	 & 24.2289 & 0.8344\\
		\hline
		DCA &  &   2.6944e+03	 & 	23.8179 & 0.8265  \\
		\InDCAe &  1.0 &  2.6974e+03  &  23.7593 & 0.8262 \\
		\RInDCAe &  &  \textbf{2.6854e+03}	 &  \textbf{24.1334} & \textbf{0.8390} \\
		ADCA &  &   2.6914e+03	 & 23.9806 &   0.8308\\
		\hline
		DCA &  &   2.8987e+03	 & 	23.3509 &  0.8202  \\
		\InDCAe &  1.1 &   2.9017e+03  &  23.2949 & 0.8171 \\
		\RInDCAe &  &  \textbf{2.8839e+03}	 & \textbf{23.7590} & \textbf{0.8433} \\
		ADCA &  &   2.8937e+03	 & 23.4928 &  0.8272\\
		\hline
		DCA &  &   3.0961e+03	 & 	23.0492 & 0.8114  \\
		\InDCAe &  1.2 &   3.0971e+03  & 22.9270 & 0.8114 \\
		\RInDCAe &  &  \textbf{3.0811e+03}	 & \textbf{23.3767} & \textbf{0.8353} \\
		ADCA &  &   3.0927e+03	 & 23.1191 &  0.8166\\
		\hline
		DCA &  &   3.2964e+03	 & 	22.4314 & 0.7942  \\
		\InDCAe &  1.3 &   3.2997e+03  & 22.4193 & 0.7933 \\
		\RInDCAe &  & \textbf{3.2745e+03}	 & \textbf{23.0360} & \textbf{0.8279} \\
		ADCA &  &   3.2908e+03	 & 22.5914 &  0.8030\\
		\hline
		DCA &  &   3.4935e+03	 & 	21.9890 & 0.7723  \\
		\InDCAe &  1.4 &   3.4967e+03  & 21.9640 & 0.7700 \\
		\RInDCAe &  &  \textbf{3.4688e+03}	 & \textbf{22.5424} & \textbf{0.8127} \\
		ADCA &  &   3.4876e+03	 & 22.0525 &  0.7797\\
		\hline
		DCA &  &   3.6881e+03	 & 	21.5175 & 0.7517  \\
		\InDCAe &  1.5 &   3.6907e+03  & 21.5080 & 0.7487 \\
		\RInDCAe &  &  \textbf{3.6584e+03}	 & \textbf{22.1562} & \textbf{0.8004} \\
		ADCA &  &   3.6817e+03	 & 21.5928 &  0.7569\\
		\hline
		DCA &  &   3.8768e+03	 & 	21.1589 & 0.7272   \\
		\InDCAe &  1.6 &   3.8791e+03  &  21.1393 & 0.7240 \\
		\RInDCAe &  &  \textbf{3.8478e+03}	 & \textbf{21.6137} & \textbf{0.7752} \\
		ADCA &  &  3.8704e+03	 & 21.2309 &  0.7357\\
		\hline
		\multicolumn{5}{l}{(Bold values for the best numerical results)}
	\end{tabular}
\end{table}

\begin{figure}[ht!]
  \subfigure[]{\includegraphics[width=0.49\textwidth]{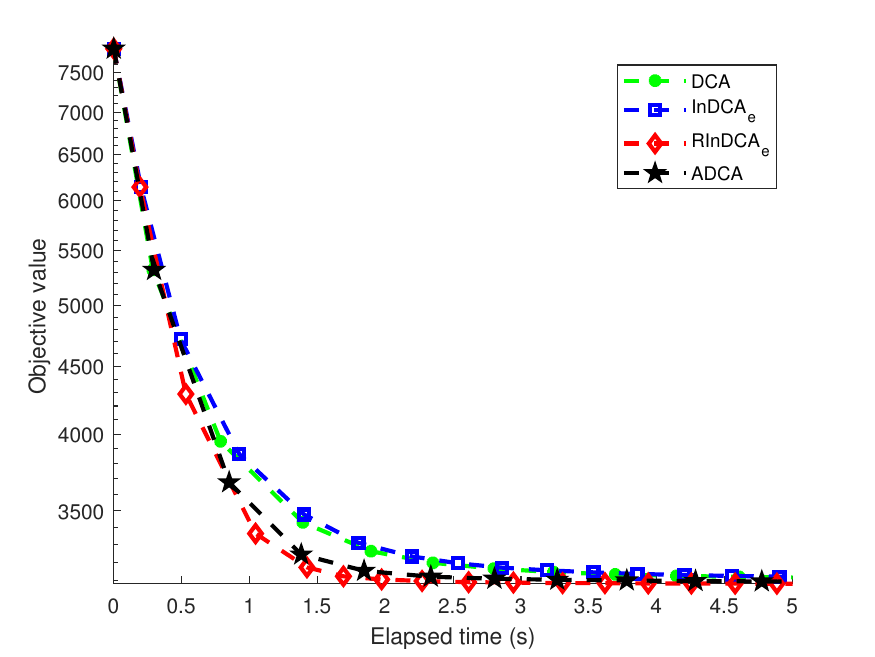}}
  \subfigure[]{\includegraphics[width=0.49\textwidth]{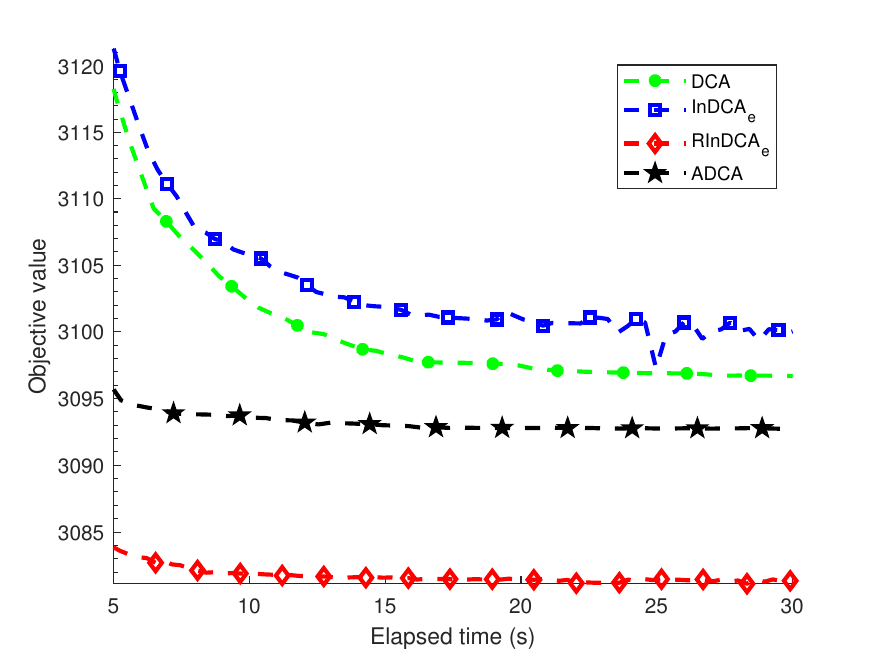}}
  \subfigure[]{\includegraphics[width=0.49\textwidth]{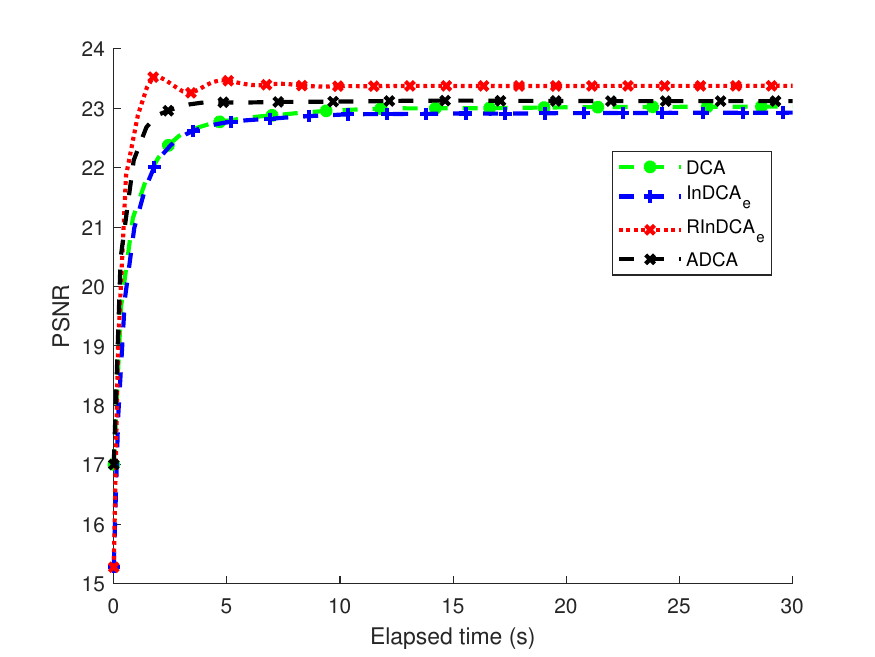}}
  \subfigure[]{\includegraphics[width=0.49\textwidth]{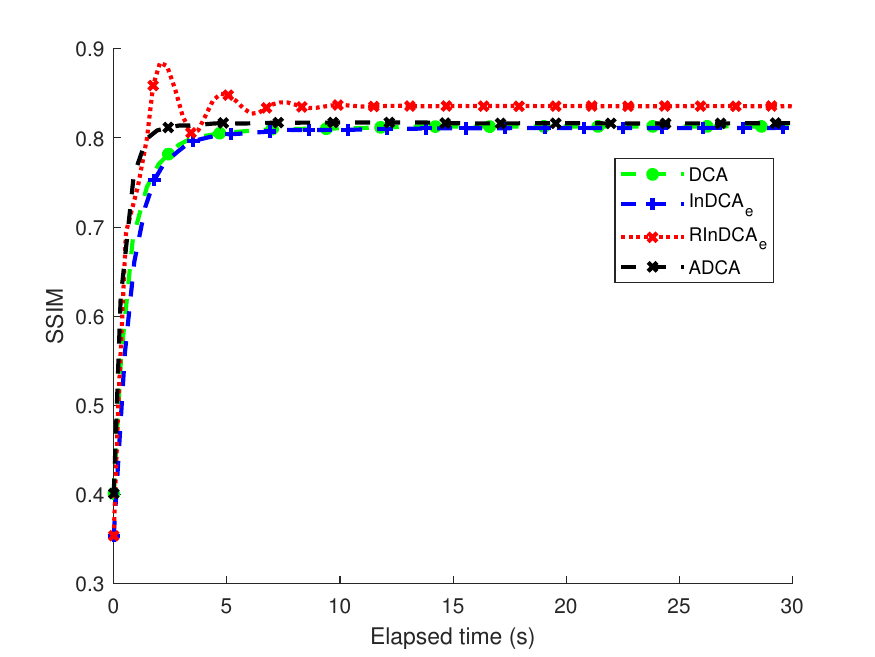}}
\caption{The trends of the objective function values together with PSNR and SSIM}
\label{fig:denoise21}      
\end{figure}

\begin{figure}[ht!]
  \subfigure[]{\includegraphics[width=0.24\textwidth]{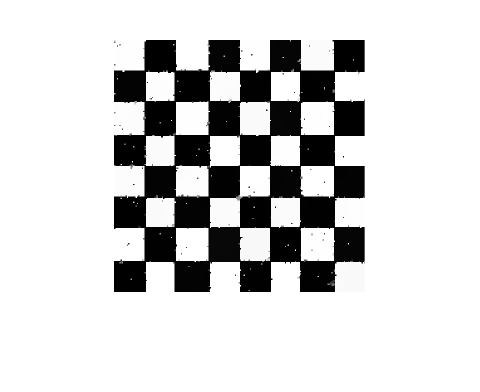}}
  \subfigure[]{\includegraphics[width=0.24\textwidth]{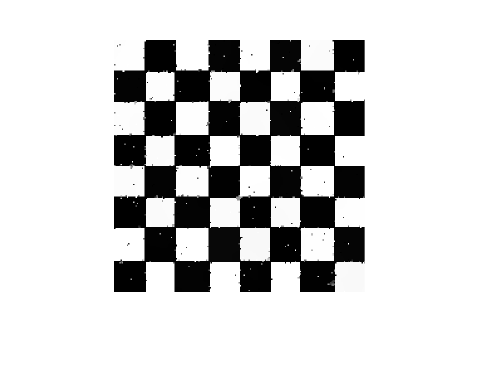}}
  \subfigure[]{\includegraphics[width=0.24\textwidth]{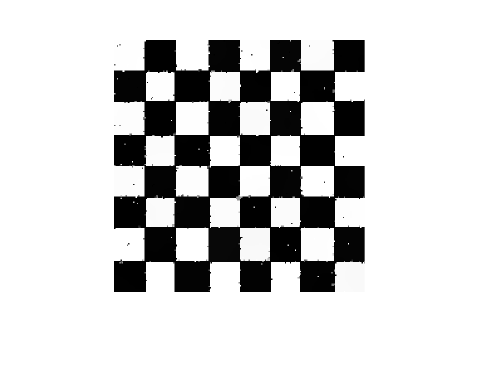}}
  \subfigure[]{\includegraphics[width=0.24\textwidth]{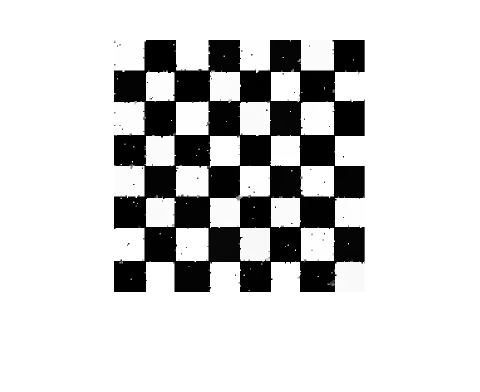}}
\caption{Image reconstructions where \textbf{a,\,b,\,c,\,d} is respectively the resulting image by DCA, \InDCAe, \RInDCAe, and ADCA} 
\label{fig:denoise22}       
\end{figure}

\section{Conclusion and perspective}
\label{sec:7}
In this paper, based on the inertial DC algorithm \cite{de2019inertial}  for DC programming, we propose a refined version with larger inertial step-size, for which we prove the subsequential convergence and the sequential convergence by further assuming the KL property. Numerical simulations on checking copositivity of matrices and image denoising problem show the good performance of our proposed methods  and the benefit of larger step-size. 
 
About the future works, the numerical test on \RInDCAine\ will be investigated. We may also introduce non-heavy-ball and non-Nesterov acceleration to DC programming.  Furthermore, the inertial-force procedure will be extended to the partial DC programming \cite{partialDC} in which the objective function $f$ depends on two variables $\vx$ and $\vy$ where $f(\vx,.)$ and $f(.,\vy)$ are both DC functions.

\begin{acknowledgements}
This work is supported by the National Natural Science Foundation of China (Grant 11601327). We thank the anonymous referees for their valuable remarks. 
\end{acknowledgements}

\bibliographystyle{spmpsci}      
\bibliography{mybib}   

\end{document}